\documentclass[final,hidelinks,onefignum,onetabnum]{siamart220329}

\usepackage{multicol}
\usepackage{cite}
\usepackage{multirow}
\usepackage{makecell}
\usepackage{pgfplots}
\pgfplotsset{compat=1.17}
\usepackage{subfig}

\usepackage{lipsum}
\usepackage{amsfonts}
\usepackage{graphicx}
\usepackage{epstopdf}
\usepackage{algorithmic}
\ifpdf
  \DeclareGraphicsExtensions{.eps,.pdf,.png,.jpg}
\else
  \DeclareGraphicsExtensions{.eps}
\fi

\usepackage{enumitem}
\setlist[enumerate]{leftmargin=.5in}
\setlist[itemize]{leftmargin=.5in}
\newsiamremark{remark}{Remark}
\newsiamremark{hypothesis}{Hypothesis}
\crefname{hypothesis}{Hypothesis}{Hypotheses}
\newsiamthm{claim}{Claim}

\headers{preconditioned ISTA for Graphical LASSO}{G. Shalom, E. Treister, and I. Yavneh}

\title{pISTA: preconditioned Iterative Soft Thresholding Algorithm for Graphical Lasso}

\author{Gal Shalom\thanks{Faculty of Computer Science, Technion — Israel Institute of Technology, Israel
  (\email{galshalom@cs.technion.ac.il}, \email{irad@cs.technion.ac.il}). Supported in part by the Israel Science Foundation, grant No. 1639/19.}
\and Eran Treister\thanks{Department of Computer Science, Ben-Gurion University of the Negev, Israel 
  (\email{erant@cs.bgu.ac.il}). This work was supported in part by the Israeli Council for Higher Education (CHE) via the Data Science Research Center, Ben-Gurion
University of the Negev, Israel.}
\and Irad Yavneh\footnotemark[1]}

\usepackage{amsopn}

\DeclareMathOperator*{\argmin}{arg\,min}
\DeclareMathOperator*{\argmax}{arg\,max}

\newcommand{\llangle}{\left\langle}
\newcommand{\rrangle}{\right\rangle}

\ifpdf
\hypersetup{
  pdftitle={pISTA: preconditioned Iterative Soft Thresholding Algorithm for Graphical Lasso},
  pdfauthor={G. Shalom, E. Treister, and I. Yavneh}
}
\fi

\usepackage{ulem}
\begin{document}

\maketitle

\begin{abstract}
    We propose a novel method for solving the sparse inverse covariance estimation problem, also known as the graphical least absolute shrinkage and selection operator (GLASSO). This problem is often solved using a second-order quadratic approximation. However, in such algorithms, the Hessian term is complex and computationally expensive to handle. Therefore, our method uses the inverse of the Hessian as a preconditioner to simplify and approximate the quadratic element at the cost of a more complex \(\ell_1\) element. The variables of the resulting preconditioned problem are coupled only by the \(\ell_1\) sub-derivative of each other, which can be guessed with minimal cost using the gradient itself, allowing the algorithm to be parallelized and implemented efficiently on GPU hardware accelerators. Numerical results on synthetic and real data demonstrate that our method is competitive with other state-of-the-art approaches. 
\end{abstract}

\begin{keywords}
  Graphical LASSO, Sparse precision matrix estimation, Proximal methods, Preconditioning. 
\end{keywords}

\begin{AMS}
  90C25, 65D18, 65K10, 65F08
\end{AMS}

\section{Introduction}
\label{sec:intro}

Inverse covariance estimation is a fundamental problem in modern statistics. Specifically, the inverse covariance matrices of multivariate normal distributions are used in numerous applications. One of the most common uses of the inverse covariance, also known as the \textit{precision matrix}, is to describe statistical models. That is, the graph structure of Gaussian graphical models can be inferred from the precision matrix \cite{intro:graph_model}. This inferred graph can be used for analyzing gene networks \cite{intro:bio}, financial assets and stocks dependencies \cite{intro:stocks,intro:stocks2}, social networks \cite{intro:social} and other inter-dependent data \cite{intro:other1,intro:other2}.

The most straightforward approach to estimating the precision matrix would be to use the inverse of the empirical covariance matrix. Evidently, inverting the empirical covariance matrix requires a number of samples that equals or exceeds the dimension of the matrix. However,  many problems of interest are high dimensional \cite{cov:BIG_QUIC,intro:bio}, to the extent that they are significantly larger than the number of available samples, so the empirical covariance matrix is rank-deficient. In such cases, some kind of regularization is essential. One of the most common ways to estimate the precision matrix is to solve a \(\ell_1\)-regularized maximum likelihood estimation problem, known as the graphical least absolute shrinkage and selection operator (Graphical LASSO, or GLASSO) problem. The use of \(\ell_1\)-regularization aims to achieve a sparse estimation while keeping the problem convex. The sparsity of the precision matrix may be interpreted as simplicity---a sparse precision matrix implies a simple inferred graph structure \cite{intro:sparse_model}. Thus, when a genuine empirical precision matrix cannot be computed, a sparse valid estimation is often the preferred choice \cite{intro:sparse_model, intro:sparse_assume,prob:cov_select}. We refer the reader to the first chapter of \cite{intro:sparse_model} and to \cite{prob:cov_select} for additional motivation. 

As the problem is well studied, many algorithms and methods for its solution have been presented over the years \cite{Sparse1,Sparse2,Sparse3,cov:GLASSO,cov:QUIC}. Some of these methods are computationally intensive or difficult to parallelize, rendering parallel computation accelerators such as GPUs ineffective. As far as we know, none of the algorithms in the literature were designed specifically with an efficient GPU deployment in mind. Nevertheless, various existing algorithms are able to take advantage of GPU processing power with a simple implementation, e.g., G-ISTA \cite{cov:GISTA}, VSM \cite{cov:VSM}, PSM \cite{cov:PSM}, ALM \cite{cov:ALM}, Newton-Lasso \cite{cov:newton_like} and Orthant-Based Newton \cite{cov:newton_like}. Each of these has its strengths and weaknesses. Other algorithms, such as BCD-IC \cite{cov:BCD}, can be deployed on a GPU as well. However, the method and the authors' implementation are complex, and the extensive usage of scalar operations may hinder its performance.

In this work we introduce a preconditioned Iterative Soft Thresholding Algorithm (pISTA) for solving the graphical LASSO problem. pISTA is an algorithm designed to be highly parallel and suitable for using GPU capabilities to its benefit by exploiting the problem structure effectively. A traditional second-order quadratic approximation includes a complex Hessian term which is computationally expensive to handle. Moreover, all the elements of the Hessian term are coupled, making efficient GPU deployment challenging. To this end, pISTA uses the inverse of the Hessian as a preconditioner to simplify and approximate the quadratic element at the cost of a more complex \(\ell_1\) element. This is highly beneficial since the Hessian inverse is easily obtained for the GLASSO problem. 
The variables of the resulting preconditioned problem are coupled only by the \(\ell_1\) sub-derivative of each other, which can be guessed with minimal cost using the gradient itself. The resulting pISTA algorithm uses only matrix operations, making it easier to exploit the GPU efficiently using current and future frameworks. On the other hand, we do not exploit the sparsity of the estimated matrix for speedup as GPUs are limited in their ability to parallelize sparse computations efficiently. They are also limited in memory. Hence we target the estimation of $n\times n$ matrices where $n$ is less than about ten thousand. Such problem sizes are relevant for many applications, like stocks, genes, brain regions, road maps, health-related measurements, etc. 

This paper is organized as follows. In \cref{sec:related} we refer to related work, and in \cref{sec:background} we formulate the problem. The pISTA algorithm is introduced in \cref{sec:pista}, and its convergence is proved in \cref{sec:analysis}. Experimental results are presented in \cref{sec:numer}, and conclusions follow in \cref{sec:conclusions}.

\section{Related Work}
\label{sec:related}

Many numerical algorithms have been developed which are specifically designed for solving the sparse inverse covariance estimation problem. G-ISTA \cite{cov:GISTA} uses a proximal gradient descent method that estimates the inverse covariance iteratively. GLASSO \cite{cov:GLASSO} splits the problem into smaller LASSO problems and updates the estimation by solving them separately. QUIC \cite{cov:QUIC} uses a proximal Newton method on the objective function. BCD-IC \cite{cov:BCD} uses a block coordinate descent method. ALM \cite{cov:ALM} uses an alternating linearization technique by splitting the objective function into two linear approximations and minimizing them alternatingly. PSM \cite{cov:PSM} employs a projected gradient method. VSM \cite{cov:VSM} uses Nesterov’s smooth
optimization technique. Newton-Lasso \cite{cov:newton_like} and Orthant-Based Newton \cite{cov:newton_like} use FISTA \cite{FISTA} and Conjugate Gradient methods, respectively, to solve reduced second-order approximations. Although some algorithms are designed to be run in parallel using the multi-core model of the CPU, none of them take into specific consideration the parallelism model of the GPU and the usage of GPUs as a computation accelerator.

Additionally, some algorithms were developed specifically for large-scale matrices where the matrix can only reside in memory in a sparse format. BIG\&QUIC and SQUIC \cite{cov:BIG_QUIC,bollhofer2019large} extend QUIC \cite{cov:QUIC} to large scales. BIG\&QUIC uses on-demand computation of the Hessian's columns and a special procedure of the linesearch conditions using Schur complements. SQUIC utilizes sparse Cholesky factors of the iterations matrix. ML-BCD \cite{cov:MultiBCD} takes BCD-IC \cite{cov:BCD}, which is already suitable to large-scale problems, and accelerates it using a generic multilevel framework, which was originally suggested in \cite{treister2012multilevel} for LASSO. This acceleration can, in principle, be used with our proposed method as well. Large-scale algorithms may benefit from the GPU processing power, however, memory constraints and usage of sparse format make this task non-trivial.

Graphical LASSO can also be used in a mixture model setup when multiple sparse inverse covariances are utilized \cite{finder2020effective}. This allows a richer statistical model but requires the estimation of multiple sparse Precision Matrices. An efficient GPU-based solver is highly beneficial in such scenarios as well.

\section{Background}
\label{sec:background}

\subsection{Sparse Inverse Covariance Estimation}
Estimating the parameters of multivariate Gaussian distributions is a fundamental problem in statistics. Given \(m\) independent samples \(\{y_i\}_{i=1}^m\in\mathbb{R}^n\), where \(y_i\sim \mathcal{N}(\mu,\Sigma)\), one would like to estimate the mean \(\mu\in\mathbb{R}^n\) and either the covariance \(\Sigma\), or its inverse \(\Sigma^{-1}\), which is also called the \textit{precision matrix}. Both the mean \(\mu\) and the covariance \(\Sigma\) are often estimated by the maximum likelihood estimator (MLE). The MLE is given by the parameters \(\mu\) and \(\Sigma\) that maximize the probability to sample the observed data \(\{y_i\}_{i=1}^m\):
\begin{equation}\label{eq:MLE_objective}
\begin{split}
    \hat{\mu},\hat{\Sigma} &= \argmax_{\mu,\Sigma}\prod^m_{i=1}\mathbb{P}(y_i|\Sigma,\mu)\\
    &= \argmax_{\mu,\Sigma}\prod^m_{i=1}\frac{1}{\sqrt{(2\pi)^m det(\Sigma)}}\text{exp}\left(-\frac{1}{2}(y_i-\mu)^T\Sigma^{-1}(y_i-\mu)\right)
    \,.
\end{split}
\end{equation}
MLE has an analytical solution:
\begin{equation}\label{eq:empirical_cov}
    \hat{\mu}=\frac{1}{m}\sum_{i=1}^m y_i, \quad\quad
    \hat{\Sigma}=S\overset{\Delta}{=}\frac{1}{m}\sum_{i=1}^m (y_i-\hat{\mu})(y_i-\hat{\mu})^T
    \,,
\end{equation}
where \(\hat{\mu}\) and \(\hat{\Sigma}\) are called the empirical mean and empirical covariance, respectively.

Usually, one considers the log-likelihood objective (negative \(log\) of \cref{eq:MLE_objective}), minimized over the inverse covariance matrix, yielding the inverse covariance MLE:
\begin{equation} \label{eq:log_MLE}
    \hat{\Sigma}^{-1} = \argmin_{A\succ0}f(A)\overset{\Delta}{=}-\log(\det(A))+Tr(SA)
    \,.
\end{equation}
The solution is indeed the inverse of \cref{eq:empirical_cov}. However, in cases where the number of available samples is smaller than the dimension of $y_i$ ($m<n$), the matrix \(S\) is rank-deficient and thus non-invertible, whereas the true \(\Sigma\) is assumed to have full rank and is positive definite. In other words, we cannot estimate \(\Sigma^{-1}\) by inverting \(S\), and further assumptions should be considered. A common choice in this case is to assume that \(\Sigma^{-1}\) is sparse. \(\Sigma^{-1}\) can be interpreted as a conditional dependence matrix, that is, its off-diagonal entries indicate the dependence between the row and column variables, given all the remaining variables \cite{rue2005gaussian}. 

A common approach is to regularize the log-likelihood objective \cref{eq:log_MLE} with a sparsity promoting \(\ell_1\)-penalty \cite{Sparse1,Sparse2,Sparse3}:
\begin{equation} \label{eq:sparse_obj}
    \hat{\Sigma}^{-1}=\argmin_{A\succ0}F(A)\overset{\Delta}{=} \argmin_{A\succ0}f(A)+\alpha||A||_1
    \,,
\end{equation}
where \(\alpha>0\) is a scalar and \(||A||_1=\sum_{i,j}|A_{i,j}|\). \Cref{eq:sparse_obj} is known as the \emph{Graphical Lasso} problem. 

As seen from \cref{eq:sparse_obj}, the objective function is convex and is composed of two parts---a smooth convex part \(f(A)\) and a non-smooth convex part \(\alpha||A||_1\). Although the objective is convex, the non-smooth term makes traditional algorithms ineffective, and more specialized solvers are needed.

\subsection{Proximal Methods for Sparse Inverse Covariance Estimation}

A common approach to solving convex problems comprised of smooth and non-smooth parts is known as \textit{proximal methods}, where we approximate the objective function at each iteration as follows: the smooth part is approximated by a quadratic function while the non-smooth part is kept unchanged. This approach is especially attractive when the non-smooth part is a separable function (e.g., point-wise). 

Specifically, applying proximal methods to the objective function \(F(A)\) of \cref{eq:sparse_obj}, the smooth term \(f(A)\) is approximated by a quadratic function while the term \(\alpha||A||_1\) remains unchanged. This approximation yields a linear LASSO \cite{LASSO} problem (\(\ell_1\)-regularized quadratic objective). Once defined, we minimize it approximately using a LASSO solver. Denoting  the descent direction at the \(k\)-th iteration by \(D^{(k)}\), then at each iteration we solve:
\begin{equation}\label{eq:proximal_gen}
    \begin{split}
        D^{(k)} &= \argmin_D \Tilde{F}(A^{(k)}+D)\\
        &= \argmin_D f(A^{(k)})+ \llangle g^{(k)},D\rrangle + \frac{1}{2}\Big\langle D,D\Big \rangle_{H^{(k)}}+\alpha||A^{(k)}+D||_1 
        \,,       
    \end{split}
\end{equation}
where \(g^{(k)}=\nabla f(A^{(k)})\) is the gradient of $f$ at the $k$-th iteration. The matrix \(H^{(k)}\) depends on the particular method, and its associated inner product is defined by 
\begin{equation}
\label{eq:HessTerm}
\Big\langle D,D \Big\rangle_{H^{(k)}} = \Big\langle \text{vec}(D),H^{(k)}\text{vec}(D)\Big\rangle, 
\end{equation}
where vec() denotes the column-stacking of a matrix to a vector. For 
\(H^{(k)}=I\) we get the proximal gradient descent method, which is called G-ISTA \cite{cov:GISTA}. On the other extreme, for 
\(H^{(k)}=\nabla^2 f(A^{(k)})\), the Hessian at the $k$-th iteration, we get the ``Proximal Newton'' method known as QUIC \cite{cov:QUIC}.

The gradient and the Hessian are given by (see Chapter A.4.3 at \cite{proof:derivation}):
\begin{equation}\label{eq:grad_hessian}
     \nabla f(A)=S-A^{-1} \quad,\quad \nabla^2 f(A)=A^{-1} \otimes A^{-1} 
     \,,
\end{equation}
where \(\otimes\) is the Kronecker product. We note that once the gradient has been computed, we can compute Hessian-vector products at a relatively low cost by using the property \((A\otimes B)\text{vec}(V)=\text{vec}(BVA^T)\). However, solving \cref{eq:proximal_gen} with the exact Hessian \(H^{(k)}=\nabla^2 f(A^{(k)})\) is complex and it is typically done by coordinate descent iterations \cite{cov:QUIC}. This makes it hard to use computation accelerators (such as GPUs) efficiently.

\section{The preconditioned Iterative Soft Thresholding Algorithm (pISTA)}
\label{sec:pista}
In this section we develop the \textit{preconditioned Iterative Soft Thresholding Algorithm} (pISTA). In \cite{cov:QUIC}, the Proximal Newton method for our problem, each coordinate computation depends on the values of all the other coordinates due to the Hessian in the quadratic element. In pISTA, we aim for a simpler and easier-to-solve quadratic element at the cost of a more complex non-smooth $\ell_1$ part. 

First, in \cref{section:free_set}, we restrict our problem to a smaller set of variables as done in other algorithms \cite{cov:QUIC,cov:MultiBCD,cov:newton_like}. In \cref{section:pre}, we simplify and relax the quadratic part using a preconditioner at the cost of making the non-smooth part \(\alpha||A||_1\) more complex. Lastly, in \cref{section:pista}, we solve the resulting problem, completing the development of pISTA. In contrast to second-order methods like QUIC, for example, in pISTA each coordinate depends on the other variables only through the sub-gradient of the non-smooth part (\(\ell_1\) norm), which is more complex. However, the sub-gradient of the \(\ell_1\) term can be approximated easily and relatively well by the sign of the elements of the current iterate (if it is nonzero) or of its gradient (where the elements of the iterate vanish), allowing our method to be efficient.

\subsection{Restricting the Updates to the Free-set} \label{section:free_set}
As introduced in \cite{cov:QUIC} and used in \cite{cov:newton_like,cov:MultiBCD}, we restrict the descent direction \(D\) to the \textit{free-set} of \(A^{(k)}\). That is, an element in \(D\) which is not in the \textit{free-set} of  \(A^{(k)}\) is set to zero. Denote by \(\mathcal{S}_A^{(k)}\) the \textit{free-set} at iteration \(k\), defined as follows:
\begin{gather}
    \label{eq:free_set}
    \mathcal{S}_A^{(k)} = \bigg\{(i,j) \big|\ A^{(k)}_{i,j}\neq 0\bigg\}\cup\bigg\{(i,j) \big|\  |\nabla f(A^{(k)}|_{i,j}>\alpha\bigg\}
\,.\end{gather}
According to Lemma 7 of \cite{cov:QUIC}, solving \cref{eq:proximal_gen} restricted to the variables that are not in the \textit{free-set} of \(A^{(k)}\) will result in a zero value in every element in the direction \(D\). Thus, restricting \(D\) to the \textit{free-set} is equivalent to solving \cref{eq:proximal_gen} in an alternating two-block manner: the first step restricted to the variables not in the \textit{free-set} (which makes no change in our approximation), and the second step restricted to the \textit{free-set}. In the works mentioned above, the free set is used to save computations. On the other hand, here we precondition the problem and couple all the unknowns in the inner quadratic problem defined using the Taylor expansion. Hence, the free set is important to maintain the correct subgradient in the linear Taylor (gradient) term, otherwise, the quadratic approximation is inconsistent with the $\ell_1$ term and the eventual update. As a result, the restriction to the free set is an integral part of our method, and this set is recomputed before each iteration. 

Define the restriction mask at the \(k\)-th iteration:
\begin{equation} \label{M}
    [\mathcal{M}_A^{(k)}]_{i,j} = \bigg\{\begin{array}{lr}
         1&  (i,j)\in\mathcal{S}_A^{(k)}\\
         0 & (i,j)\notin\mathcal{S}_A^{(k)}
    \end{array} 
\,.\end{equation}
Then, the restricted problem we solve at the $k$th iteration can be written as:
\begin{equation} \label{eq:proximal_gen_free}
    \begin{split}
        D^{(k)} &= \argmin_D \Tilde{F}(A^{(k)}+(\mathcal{M}_{A}^{(k)} \odot D))
    \end{split}
\,,\end{equation}
where \(\odot\) is the Hadamard product. Note that, in the solution of \cref{eq:proximal_gen_free}, the elements of \(D^{(k)}\) that are not in the \textit{free-set} of \(A^{(k)}\) may have nonzero values. However, we restrict the updates to the indices in the \textit{free-set} of \(A^{(k)}\) only. Therefore, our descent direction is \(\mathcal{M}_A^{(k)}\odot D^{(k)}\), which is equivalent to setting the values of \(D^{(k)}\) that are not in the free-set to zero.

\subsection{Preconditioning of the Descent Direction} \label{section:pre}
Note the following property of the Hessian in \cref{eq:grad_hessian}:
\begin{equation}
    \left(\nabla^2f(A)\right)^{-1} = A \otimes A
\,,\end{equation}
where \(\otimes\) is the Kronecker product. This means that the Hessian inverse can be obtained without any computational overhead, using $A$ instead of $A^{-1}$. Here we are interested in using this appealing property to accelerate the solution of \cref{eq:proximal_gen}, where \(H^{(k)}=\nabla^2 f(A^{(k)})\), using the Hessian as a preconditioner. To this end, we define a new variable \(\Delta\) and define $D$ by:
\begin{equation} \label{eq:D_hess}
    \text{vec}(D) = \big(\nabla^2f(A^{(k)})\big)^{-1}\text{vec}(\Delta)\Rightarrow D = A^{(k)}\Delta A^{(k)}
    \,,
\end{equation}
where \(\text{vec}()\) denotes the column-stacking of a matrix into a vector. We note that because \(A^{(k)}\) is positive definite, it is of full rank, and the fact that $\Delta$ is not constrained implies that $D$ is not restricted to a specific subspace.

Denote \(W^{(k)}=(A^{(k)})^{-1}\). Using the preconditioned descent direction \cref{eq:D_hess} in the restricted problem \cref{eq:proximal_gen_free}, with \(H^{(k)}=\nabla^2f(A^{(k)}) = W^{(k)}\otimes W^{(k)}\), results in the following equation:
\begin{equation}\label{eq:pre_newton}
    \begin{split}
        \Delta^{(k)} &= \argmin_\Delta \Tilde{F}\left(A^{(k)}+\mathcal{M}_A^{(k)} \odot (A^{(k)}\Delta A^{(k)})\right)\\
        &= \argmin_\Delta f(A^{(k)}) + \llangle g^{(k)},\mathcal{M}_A^{(k)} \odot (A^{(k)}\Delta A^{(k)})\rrangle\\
        &\quad\quad+ \frac{1}{2}\llangle\mathcal{M}_A^{(k)} \odot (A^{(k)}\Delta A^{(k)}),W^{(k)}\left(\mathcal{M}_A^{(k)} \odot (A^{(k)}\Delta A^{(k)})\right)W^{(k)}\rrangle\\
        &\quad\quad+ \alpha\left|\left|A^{(k)}+\mathcal{M}_A^{(k)} \odot (A^{(k)}\Delta A^{(k)})\right|\right|_1        
    \,.
    \end{split}
\end{equation}

After applying the mask restriction \(\mathcal{M}\) and the preconditioning to \cref{eq:proximal_gen}, the quadratic part (the second element in \cref{eq:pre_newton}) is still complex. We relax it by removing the mask restriction \(\mathcal{M}\) and replacing it with a multiplicative scalar:
\begin{multline}\label{eq:relax}
    \frac{1}{2}\llangle\mathcal{M}_A^{(k)} \odot (A^{(k)}\Delta A^{(k)}),W^{(k)}\left(\mathcal{M}_A^{(k)} \odot (A^{(k)}\Delta A^{(k)})\right)W^{(k)}\rrangle\\
    \approx\frac{1}{2t}\llangle A^{(k)}\Delta A^{(k)},W^{(k)}(A^{(k)}\Delta A^{(k)})W^{(k)}\rrangle=\frac{1}{2t}\llangle A^{(k)}\Delta A^{(k)},\Delta \rrangle
\,,\end{multline}
where \(t\) is a scalar that is computed using linesearch. Since we do not change the first-order terms (the gradient and \(\ell_1\)), this can still result in a monotonically convergent method. Denote the resulting approximation by \(P(\Delta;t)\), then the \(k\)-th iterate is given by:
\begin{equation}\label{eq:eq_quasi_newton}
    \begin{split}
        \Delta^{(k)} &= \argmin_\Delta P(\Delta;t)\\
        &= \argmin_\Delta f(A^{(k)}) + \llangle g^{(k)},\mathcal{M}_A^{(k)} \odot (A^{(k)}\Delta A^{(k)})\rrangle+ \frac{1}{2t}\llangle A^{(k)}\Delta A^{(k)},\Delta \rrangle\\
        &\quad\quad+ \alpha\left|\left|A^{(k)}+\mathcal{M}_A^{(k)} \odot (A^{(k)}\Delta A^{(k)})\right|\right|_1        
\,.
\end{split}
\end{equation}

\subsection{The pISTA Method} \label{section:pista}
Our pISTA algorithm solves the problem defined by \cref{eq:eq_quasi_newton} in each iteration, with an appropriately chosen \(t\). First, we find a solution \(\Delta^{(k)}\) to \cref{eq:eq_quasi_newton}. To obtain this, we need to formulate the sub-differential of the quadratic approximation $P$. First, denote the sub-differential of the non-smooth \(\ell_1\) term by:
\begin{eqnarray} \label{eq:l1_grad}
    G^{(k)}(D) &\overset{\Delta}{=}&  \frac{\partial ||A^{(k)}+D||_1}{\partial D} \\ \nonumber &=& \Bigg\{
   \mathcal{G}^{(k)} \Bigg|\begin{array}{lr}
         \mathcal{G}^{(k)}_{i,j}=[sign(A^{(k)}+D)]_{i,j}&  [A^{(k)}+D]_{i,j}\neq 0\\
         \mathcal{G}^{(k)}_{i,j}\in[-1,1]& [A^{(k)}+D]_{i,j}=0
\end{array} \Bigg\}.
\end{eqnarray}
Next, using the chain rule and some known derivatives with \eqref{eq:l1_grad}, we formulate the sub-differential of $P$:
\begingroup\small
\begin{multline} \label{eq:quasi_grad}
    \frac{\partial P(\Delta;t)}{\partial \Delta} = \Bigg\{
    \frac{1}{t} A^{(k)}\Delta A^{(k)} +
    A^{(k)}\left(g^{(k)} \odot \mathcal{M}_A^{(k)}\right)A^{(k)}+\alpha A^{(k)}\left(\mathcal{G}^{(k)}\odot \mathcal{M}_A^{(k)}\right)A^{(k)} \\\Bigg|\ \mathcal{G}^{(k)}\in G^{(k)}\left(\mathcal{M}_A^{(k)} \odot (A^{(k)}\Delta A^{(k)})\right) \Bigg\}
\,,\end{multline}\endgroup
where \(\mathcal{G}^{(k)}\) represents a sub-gradient of the \(\ell_1\) term. 


The desired \(\Delta^{(k)}\) that will be used to compute the descent direction is one which includes the matrix 0 in its sub-gradients: 
\begin{equation} \label{eq:delta_grad}
    \Delta^{(k)}:\quad\quad 0\in\frac{\partial P(\Delta^{(k)};t)}{\partial \Delta}
\,.\end{equation}

Notice that after finding \(\Delta^{(k)}\), it will be used to compute the descent direction \(\mathcal{M}_A^{(k)}\odot D^{(k)}=\mathcal{M}_A^{(k)}\odot(A^{(k)}\Delta^{(k)}A^{(k)})\).  Thus, we can equivalently find \(D^{(k)}=A^{(k)}\Delta^{(k)}A^{(k)}\) instead of \(\Delta^{(k)}\). In other words, we shall find \(D^{(k)}\) which has 0 as one of its sub-gradients:
\begin{multline} \label{eq:d_grad}
    D^{(k)}: 0\in\Bigg\{
    \frac{1}{t}D^{(k)} +
    A^{(k)}\left(g^{(k)} \odot \mathcal{M}_A^{(k)}\right)A^{(k)}+\alpha A^{(k)}\left(\mathcal{G}^{(k)}\odot \mathcal{M}_A^{(k)}\right)A^{(k)} \\\Bigg|\ \mathcal{G}^{(k)}\in G^{(k)}\left(\mathcal{M}_A^{(k)} \odot D^{(k)}\right) \Bigg\}
\,.\end{multline}
Note that the change of variables has made the first (Hessian) term trivial. To find \(D^{(k)}\), we split the sub-gradient to \(n^2\) equations where the \((i,j)\)-th equation is composed by the \((i,j)\) element in the sub-gradient. For each equation \((i,j)\), we define the equation variable as \(D^{(k)}_{i,j}\). The approximate solution (developed in \cref{appendix:d_sol}) is:
\begin{equation}\label{eq:d_sol}
    D^{(k)}_{i,j}=-A^{(k)}_{i,j}+\text{SoftThreshold}\left(A^{(k)} - t\cdot B^{(k)}_{i,j}\ ,\ t\cdot C^{(k)}_{i,j}\right)
    \,,
\end{equation}
where\begin{equation} \label{eq:c}
    C^{(k)}_{i,j}=\left\{\begin{array}{lc}
        \alpha\cdot\left(A^{(k)}_{i,i}\cdot A^{(k)}_{j,j}\right), & i= j\\
        \alpha\cdot\left(A^{(k)}_{i,i}\cdot A^{(k)}_{j,j}+A^{(k)}_{i,j}\cdot A^{(k)}_{j,i}\right), & i\neq j
    \end{array}\right.
    \,,
\end{equation}
and
\begin{equation} \label{eq:b}
    B^{(k)}=
    A^{(k)}\left(g^{(k)} \odot \mathcal{M}_A^{(k)}\right)A^{(k)} + \alpha A^{(k)}\left(\mathcal{G}^{(k)}\odot \mathcal{M}_A^{(k)}\right)A^{(k)}-C^{(k)}\odot \left(\mathcal{G}^{(k)}\odot \mathcal{M}_A^{(k)}\right)
    \,.
\end{equation}

So far, we developed a second-order approximation for the GLASSO using the Hessian inverse and a quadratic relaxation. We found a closed form direction \(D^{(k)}\) for each \((i,j)\) entry, assuming that the rest of the entries in \(\mathcal{G}^{(k)}\) are given. In the following subsections we show how we approximate the entries of \(\mathcal{G}^{(k)}\) and how to find an appropriate step size \(t\).  



\subsubsection{The Approximation of \(\mathcal{G}\)}
The algorithm depends on a good approximation of \(\mathcal{G}^{(k)}\), especially at the beginning of the solution process. Recall that \(\mathcal{G}^{(k)}_{i,j}\) is the sub-derivative of \(|{A}^{(k)}_{i,j} + {D}^{(k)}_{i,j}|\), and although the sub-derivative cannot be computed, it can be approximated by \(sign({A}^{(k)}_{i,j} + {D}^{(k)}_{i,j})\). Eventually, at the late iterations \(\mathcal{G}^{(k)}\odot \mathcal{M}_A^{(k)}\) will converge to \(sign(A^{(k)})\), as all zero elements in \(A^{(k)}\) will be out of the free-set \(\mathcal{S}^{(k)}_A\). Therefore, their elements in \(\mathcal{G}^{(k)}\odot \mathcal{M}_A^{(k)}\) will be zero, and all non-zero elements in \(A^{(k)}\) will have a constant \(\mathcal{G}^{(k)}\), which is their sign. Moreover, at those later iterations, where \(sign(A^{(k)})\) is not expected to change at all, the above formula \cref{eq:d_sol} solves \cref{eq:d_grad} completely as all the terms except $D^{(k)}$ are constant. On the other hand, at the initial iterations \(\mathcal{G}^{(k)}\odot \mathcal{M}_A^{(k)}\) cannot be easily computed because some elements in the free set may change sign or be zero in the final solution. To this end, we approximate \(\mathcal{G}^{(k)}\) by predicting the ``next'' sign of the elements. That is, we take the sign of the non-zero entries, and if an entry is zero, then we take the opposite sign of its gradient. Thus, a good approximation of \(\mathcal{G}^{(k)}\) is:
\begin{equation} \label{G+}
    \mathcal{G}^{(k)}_{i,j}=\left\{\begin{array}{lr}
         sign(A^{(k)}_{i,j})& A^{(k)}_{i,j}\neq0 \\
         -sign(g^{(k)}_{i,j})&  A^{(k)}_{i,j}=0
    \end{array}\right.
\,,\end{equation}
where \(g^{(k)}=\nabla f(A^{(k)})\) as defined before. This means that we choose the next signs according to a small step of gradient descent. The above \(\mathcal{G}^{(k)}\) together with the free-set restriction mask \(\mathcal{M}^{(k)}\) define the minimum sub-gradient of \(F(A^{(k)})\) with respect to \(\ell_2\)-norm.

\subsubsection{Performing Linesearch to find \(t\)}
As proved later in \cref{sec:analysis}, for the algorithm to converge, it is sufficient for the parameter \(t\) in \eqref{eq:d_sol} to satisfy:
\begin{gather}
A^{(k)}+\mathcal{M}_A^{(k)}\odot D^{(k)} \succ 0\,, \label{t_cond_pd}\\
    F(A^{(k)}+\mathcal{M}_A^{(k)}\odot D^{(k)}) < F(A^{(k)}) \label{t_cond_min}
\,.\end{gather}
Therefore, at each iteration we perform linesearch over \(t\) until  \cref{t_cond_pd} and \cref{t_cond_min} are satisfied. As we prove later in \cref{sec:analysis_min}, a \(t\) which satisfies \(t<(\frac{\beta}{\gamma})^2\) is guaranteed to be appropriate, where \(\beta\) and \(\gamma\) are lower and upper bounds of the eigenvalues, respectively. This means that a backtracking linesearch scheme of dividing \(t\) by a constant number at each step is guaranteed to end within a finite number of steps.

\subsubsection{Summary of the pISTA Algorithm}
\Cref{alg:pista_impl} describes the full pISTA algorithm which we implemented. It is important to note that all the steps and operations in \cref{alg:pista_impl} can be computed using standard matrix subroutines, which makes our code easy to follow, flexible for future updates, and portable between computer systems. This also allows us to easily benefit from accelerated GPU computation and easy deployment using public math libraries with GPU backend. However, current GPU math libraries are highly optimized for dense operations, and the sparse kernels struggle to exploit the sparsity in the problem. That is because sparse operations involve additional memory accesses (for the indices) and irregular memory access patterns, creating computational overhead. As such, our implementation uses dense operations only and is suitable for small and medium-sized problems, e.g., up to hundreds or thousands of variables. In such problem sizes, the sparsity levels are typically not significant enough to benefit in computing time compared to dense operations. Large-sized problems, e.g. in the millions of variables like the ones considered in SQUIC \cite{bollhofer2019large}, BIG\&QUIC \cite{cov:BIG_QUIC} and BCD-IC \cite{cov:BCD}, must exploit the sparsity of the problems, which is crucial in that case, not only in terms of memory complexity but also in terms of run-time complexity.

Because we consider small to medium-sized problems in \cref{sec:numer}, we mostly compare pISTA to algorithms that use dense operations such as GISTA \cite{cov:GISTA}, Orthant-Based Newton (OBN) \cite{cov:newton_like} and ALM \cite{cov:ALM}, which are implemented on the GPU in the same manner. As stated above, the usage of dense operations for the computation of \(A^{-1}\) and \(det(A)\) in each iteration results in \(O(n^3)\) run time complexity of pISTA and the other methods for each iteration. We rely on the efficiency with which dense operations are applied but consider medium-sized problems, as the theoretical complexity is high. It should be noted, in this context, that such problems of medium size are common in many real-life scenarios.

The GPU implementation of pISTA relies on the internal parallelism of public math libraries and does not take into consideration multi-threading. However, our implementation avoids redundant copies of data from the system memory to the GPU memory and vice versa by applying as much of the computations on the GPU, carefully managing the relatively small GPU memory.

The CPU counterpart implementation follows the same high-level code as the GPU counterpart, only using the CPU backend of the standard public math libraries. Similarly to the GPU implementation, we rely on the internal parallelism of the public libraries and do not take additional multi-threading into consideration. For completeness, in \cref{sec:numer} we also compare our CPU version of pISTA to the available CPU implementations of SQUIC \cite{bollhofer2019large} and BIG\&QUIC \cite{cov:BIG_QUIC}.

\begin{algorithm}[h]
\caption{pISTA(\(S,\alpha,A^{(0)}\))}
\label{alg:pista_impl}
\begin{algorithmic}[1]
\STATE{\textbf{Result:} \(A\)}
\STATE{\textbf{Initialization:} \(k=0\);}
\WHILE{stop criteria not met}
    \STATE{Compute \(\mathcal{S}_A^{(k)}\) according to \cref{eq:free_set};}
    \STATE{Compute \(\mathcal{M}_A^{(k)}\) according to \cref{M};}
    \STATE{Compute \(\mathcal{G}^{(k)}\) according to \cref{G+};}
    \FOR{\(t\) in \(1,\ldots,\epsilon\)}
        \STATE{Compute \(D^{(K)}\) according to \cref{eq:d_sol};\\
        \(A^{(k+1)}=A^{(K)}+D^{(K)}\);}
        \IF{(\(A^{(k+1)}\succ 0 \)) \textbf{AND} (\(F(A^{(k+1)})< F(A^{(k)}\)))}
        \STATE{break;}
        \ENDIF
    \ENDFOR
\ENDWHILE
\end{algorithmic}
\end{algorithm}

\section{Convergence Analysis}
\label{sec:analysis}
Throughout this section, unless stated otherwise, all the matrices are assumed to be symmetric and real. The definitions of \(f(A)\), \(F(A)\), \(P(\Delta;t)\), \(g\), \(\mathcal{M}\), \(G\) and \(D\) are as above. We will denote \(D_t^{*(k)}\) as the direction \(D^{(k)}\) which satisfies equation \cref{eq:d_grad} for a given \(t\): 
\begin{multline} \label{eq:d_grad2}
    0\in\Bigg\{
    \frac{1}{t}D_t^{*(k)} +
    A^{(k)}\left(g^{(k)} \odot \mathcal{M}_A^{(k)}\right)A^{(k)}+\alpha A^{(k)}\left(\mathcal{G}^{(k)}\odot \mathcal{M}_A^{(k)}\right)A^{(k)} \\\Bigg|\ \mathcal{G}^{(k)}\in G^{(k)}\left(\mathcal{M}_A^{(k)} \odot D_t^{*(k)}\right) \Bigg\}
\,.\end{multline}
Also, denote by \(\mathcal{G}_t^{*(k)}\) the \(\mathcal{G}^{(k)}\) for which \(D_t^{*(k)}\) attains 0:
\begin{equation}  \label{eq:d_grad3}
    0=\frac{1}{t}D_t^{*(k)} +
    A^{(k)}\left(g^{(k)} \odot \mathcal{M}_A^{(k)}\right)A^{(k)}+\alpha A^{(k)}\left(\mathcal{G}_t^{*(k)}\odot \mathcal{M}_A^{(k)}\right)A^{(k)}
\,.
\end{equation}
In the proofs only, we will omit the iteration symbol when the meaning is evident. 

The rest of this section is organized as follows. In \cref{sec:analysis_fixed_point} we prove that if our pISTA iteration does not update \(A^{(k)}\), then \(A^{(k)}=\argmin_{A\succ0}F(A)\). 
Following that, in \cref{sec:analysis_min} we prove that for every iteration, there exists such a \(t\geq t_{min} >0\) that satisfies \cref{t_cond_pd}-\cref{t_cond_min}, i.e., keeps the matrix positive definite and decreases our objective value. Combining all the theorems and lemmas, we get the following:
\begin{corollary}
By \cref{coro:dec_positive}, \cref{lemma:fixed}, \cref{lemma4.1}, \cref{lemma:BCD4.3} and \cref{coro:compact}, pISTA converges to \(\argmin_{A\succ0}F(A)\).
\end{corollary}

\subsection{Fixed Point Iteration at the Minimum}
\label{sec:analysis_fixed_point}
The following \cref{lemma:fixed} states that if a pISTA iteration ends with no update to \(A^{(k)}\), then \(A^{(k)}=\argmin_{A\succ0}F(A) \).

Let us first note that the $-\log\det(A)$ term in \cref{eq:log_MLE} also serves as a barrier function for the constraint $A\succ 0$. That is, as \(A\) gets closer to the domain's boundary (closer to being singular), \(\det(A)\to0 \) and therefore, \(F(A)\to\infty\) due to the $\log$ term. Because of that, we can say that there exists some $\delta > 0$ for which the minimizer \(A^*\succeq \delta I\) and we can state that our optimality condition is given by
\begin{equation}\label{eq:min_sub_zero}
0 \in \frac{\partial F(A^{*})}{\partial A} \iff A^{*}=\argmin_{A\succ0}F(A).
\end{equation}
Furthermore, we can state the following corollary:
\begin{corollary} \label{coro:dec_positive}
    Due to the continuity of the objective function \(F(A)\), any decrease in the objective keeps \(A\) positive definite.
\end{corollary}

First, the following Lemma states that if \(A\) is the minimum of \(F(\cdot)\), then pISTA will result in no update.
\begin{lemma} \label{lemma:fixed}
\(D_t^{*(k)}=0 \iff A^{(k)}=\argmin_{A\succ0}F(A)\).
\end{lemma}
\begin{proof}
 Assume \(A^{(k)}=\argmin_{A\succ0}F(A)\), therefore, according to \cref{eq:min_sub_zero}, 0 is one of the sub-gradients of \(F(A)\) at \(A^{(k)}\):
\begin{equation}
    0 \in \frac{\partial F(A^{(k)})}{\partial A} = \Bigg\{
    g+\alpha \mathcal{G} \Bigg| \begin{array}{lr}
         \mathcal{G}_{i,j}=sign(A_{i,j})&  A_{i,j}\neq 0\\
         \mathcal{G}_{i,j}\in[-1,1]& A_{i,j}=0
    \end{array} \Bigg\} = \Bigg\{g+\alpha \mathcal{G}\ \Bigg|\ \mathcal{G}\in G(0) \Bigg\}
\,.\end{equation}
Let \(\mathcal{G}^{*}\) be the sub-gradient \(\mathcal{G}\) for which \(g+\alpha \mathcal{G}^{*}=0\). Thus:
\begin{align*}
    &g+\alpha \mathcal{G}^{*}=0\\
    \Rightarrow& (g+\alpha \mathcal{G}^{*}) \odot \mathcal{M}_A = 0\\
    \Rightarrow& A\left((g+\alpha \mathcal{G}^{*}) \odot \mathcal{M}_A\right)A = 0\\
    \Rightarrow& A\left(g\odot \mathcal{M}_A\right)A+\alpha A \left(\mathcal{G}^{*}\odot \mathcal{M}_A\right)A = 0
\,.\end{align*}
If we write this in a sub-differential form, we get:
\begin{equation*}
    0 \in \Bigg\{
    A\left(g\odot \mathcal{M}_A\right)A+\alpha A \left(\mathcal{G}\odot \mathcal{M}_A\right)A\ \Bigg|\  \mathcal{G}\in G(0) \Bigg\}
\,,\end{equation*}
which, according to \cref{eq:d_grad2}, means that \(D_t^{*}=0\) has a 0 in its sub-gradients.

Now, we will prove that if a pISTA iteration ends with no update, then \(A\) is the \(F(\cdot)\) minimum. Assume \(D_t^{*}=0\), then, according to \cref{eq:d_grad3}, \(\mathcal{G}_t^{*}\) satisfies:
\begin{align*}
    & A\left(g\odot \mathcal{M}_A\right)A+\alpha A \left(\mathcal{G}_t^{*}\odot \mathcal{M}_A\right)A = 0\\
    \Rightarrow & A\left((g +\alpha \mathcal{G}_t^{*})\odot \mathcal{M}_A\right)A = 0\\
    \Rightarrow & (g +\alpha \mathcal{G}_t^{*})\odot \mathcal{M}_A = 0
\,.\end{align*}

Let us now define \(\hat{\mathcal{G}}\) so that \(g +\alpha \hat{\mathcal{G}} \in \frac{\partial F}{\partial A}\):
\begin{gather}
    \hat{\mathcal{G}}_{i,j}=\Bigg\{\begin{array}{lr}
         [\mathcal{G}^{*}_t]_{i,j}&  [\mathcal{M}_A]_{i,j}\neq 0 \\
         \left[ g/\alpha\right]_{i,j} & [\mathcal{M}_A]_{i,j}=0
    \end{array}
\,.\end{gather}
It is clear that \(g +\alpha \hat{\mathcal{G}} = 0\), because if \([\mathcal{M}_A]_{i,j}=0\), then by definition \(A_{i,j}=0\) and \(|g|\leq\alpha\). This implies that \(A^{(k)}=\argmin_{A\succ0}F(A)\).
\end{proof}
\noindent In the proof of \cref{lemma:fixed}, no assumptions were made on \(t\), hence it holds for any $t$.

\subsection{Decrease in the Objective Function}
\label{sec:analysis_min}
In this section we show that under suitable conditions, the sequence \(\{F(A^{(k)})\}\) converges to \(\min_{A\succ0}F(A)\), which means that the sequence \(\{A^{(k)}\}\) converges to \(\argmin_{A\succ0}F(A)\). Then, we prove that our pISTA iteration satisfies those conditions, which completes the convergence proof.
In this section, we denote the space of symmetric positive definite
matrices by \(S_{++}\).

Before proving our main result, we state a few auxiliary lemmas. We 
first prove that \(f(A)\) has a Lipschitz continuous gradient under suitable conditions and state an important property for this kind of function.
\begin{lemma} \label{lemma:lipsithc_function_f}
Let \(\beta\) be a positive constant, then the function \(f(A)\) has a Lipschitz continuous gradient over the domain \(B=\{X\in S_{++}\big|\beta I \preceq X \preceq \gamma I\}\) with a Lipschitz constant of \(\frac{1}{\beta^2}\). As a result we get that for every \(Y,X\in B\):
\begin{equation}\label{eq:Taylor_bound}
    f(Y) \leq f(X) + \llangle \nabla f(X), (Y-X)\rrangle+\frac{1}{2\beta^2}\llangle Y-X, Y-X \rrangle.
\end{equation}
\end{lemma}
\begin{proof}
To prove that the function \(f(A)\) has a Lipschitz continuous gradient with a Lipschitz constant of \(\frac{1}{\beta^2}\), we need to prove that \(||\nabla f(X) - \nabla f(Y)||_F \leq  \frac{1}{\beta^2}||X-Y||_F\) for every \(X,Y\in B\). The gradient is given by \(\nabla f(A) = S - A^{-1}\), hence \(\forall X,Y\in B\):
\begin{align*}
    ||\nabla f(X) - \nabla f(Y)||_F &= ||Y^{-1} - X^{-1}||_F\\
    &=||Y^{-1}(X - Y)X^{-1}||_F\\
    &\leq ||Y^{-1}||_2 ||X-Y||_F ||X^{-1}||_2\\
    &= \frac{1}{\lambda_{min}(Y)}\frac{1}{\lambda_{min}(X)}||X-Y||_F\\
    &\leq \frac{1}{\beta^2}||X-Y||_F
\,,\end{align*}
where in the third inequality $\|\cdot\|_2$ denotes the induced $\ell_2$ norm. From here, \eqref{eq:Taylor_bound} follows immediately. 
\end{proof}

The rest of the proofs are inspired by \cite{cov:MultiBCD}, \cite{cov:QUIC} and \cite{proof:sparse_recon}, following the proofs of \cite{cov:MultiBCD} with proper adaptations according to our algorithm.
\begin{lemma} [corresponds to lemma 4.1 in \cite{cov:MultiBCD}] \label{lemma4.1}
Let \(\beta\) be a positive constant and \(B=\{X\in S_{++}\big|\beta I \preceq X \preceq \gamma I\}\). Assume that for every iteration \(k\), \(A^{(k)}\in B\), then:
\begin{equation}\label{eq4.1}
    F(A^{(k)}) - F(A^{(k+1)}) \geq L \cdot ||A^{(k)} - A^{(k+1)}||^2_F  
\,.\end{equation}
Where \(L>0\). Furthermore, the linesearch parameter \(t\) can be chosen to be bounded away from zero, i.e., \(t\geq t_{min}>0\).
\end{lemma}

\begin{proof}
According to pISTA iteration, \(A^{(k+1)} = A^{(k)}+\mathcal{M}_A^{(k)}\odot D_t^{*(k)}\). Therefore,
\begin{equation*}
    \begin{split}
        F(A^{(k)}) - F(A^{(k+1)}) 
        & = F(A^{(k)}) - F(A^{(k)}+\mathcal{M}_A^{(k)}\odot D_t^{*(k)})
    \end{split}
\,,\end{equation*}
and from now on, we shall omit the iteration symbol \(k\):
\begin{equation*}
    \begin{split}
        F(A^{(k)}) - F(A^{(k+1)}) 
        & = F(A) - F(A+\mathcal{M}_A\odot D_t^{*})\\
        & = F(A) - f(A+\mathcal{M}_A\odot D_t^{*}) - \alpha||A +\mathcal{M}_A\odot D_t^{*}||_1
\,.    \end{split}
\end{equation*}
Using \cref{lemma:lipsithc_function_f} on \(f(A+\mathcal{M}_A\odot D_t^{*})\):
\begin{equation}
\label{eq:LinEq1}
    \begin{split}
        F(A^{(k)}) - F(A^{(k+1)})
        &\geq  \alpha||A||_1 - \Big(\llangle g,\mathcal{M}_A\odot D_t^{*}\rrangle \\
        &\quad + \frac{1}{2\beta^2}\llangle\mathcal{M}_A\odot D_t^{*},\mathcal{M}_A\odot D_t^{*}\rrangle+\alpha||A+\mathcal{M}_A\odot D_t^{*}||_1\Big)\\
        &\geq  \alpha||A||_1 - \Big(\llangle g,\mathcal{M}_A\odot D_t^{*}\rrangle\\
        &\quad + \frac{1}{2\beta^2}\llangle D_t^{*},D_t^{*}\rrangle+\alpha||A+\mathcal{M}_A\odot D_t^{*}||_1\Big).
        \end{split}
\end{equation}
By the definition of \(P(\Delta;t)\) in \cref{eq:eq_quasi_newton} and setting $D_t^*=A\Delta_t^*A$:
\begin{equation}
\label{eq:LinEq2}
    \begin{split}
        \llangle g,\mathcal{M}_A\odot D_t^{*}\rrangle = P(\Delta^{*}_t;t) &- \alpha||A+\mathcal{M}_A\odot D_t^{*}||_1\\
        &-\frac{1}{2t}\llangle D_t^{*},A^{-1}D_t^{*}A^{-1}\rrangle-f(A)
        \,.    \end{split}
\end{equation}
Therefore, by \cref{eq:LinEq1} and \cref{eq:LinEq2}, using $F(A)=f(A)+\alpha\|A\|_1$ we get 
\begin{equation*}
    \begin{split}
     F(A^{(k)}) - F(A^{(k+1)})
        & \geq F(A) - \frac{1}{2\beta^2}\llangle D_t^{*},D_t^{*}\rrangle \\
        &\quad+\frac{1}{2t}\llangle D_t^{*},A^{-1} D_t^{*}A^{-1}\rrangle - P(\Delta_t^{*};t)
        \,.    \end{split}
\end{equation*}

\noindent Because \({\Delta}_t^{*}=\argmin_{\Delta} P(\Delta;t)\), we know that \(P(\Delta_t^{*};t)\leq P(0;t)=F(A)\):
\begin{equation*}
    \begin{split}
     F(A^{(k)}) - F(A^{(k+1)})
        & \geq F(A) - \frac{1}{2\beta^2}\llangle D_t^{*}, D_t^{*}\rrangle 
        +\frac{1}{2t}\llangle D_t^{*},A^{-1} D_t^{*}A^{-1}\rrangle - P(0;t)\\
        &\geq -\frac{1}{2\beta^2}\llangle D_t^{*}, D_t^{*}\rrangle
        +\frac{1}{2t}\llangle D_t^{*},A^{-1} D_t^{*}A^{-1}\rrangle\\
        &\geq (\frac{1}{2t\gamma^2}-\frac{1}{2\beta^2})\llangle D_t^{*},D_t^{*}\rrangle
\,.    \end{split}
\end{equation*}

Denote \(L=\frac{1}{2t\gamma^2}-\frac{1}{2\beta^2}\), then \(L\) is positive as long as \(t<(\frac{\beta}{\gamma})^2\). Furthermore, \(t\) can be chosen to be bounded
away from zero, e.g., \(t \geq t_{min}  = \frac{1}{2}\cdot(\frac{\beta}{\gamma})^2 > 0\).
\end{proof}

\begin{lemma} [corresponds to lemma 4.2 on \cite{cov:MultiBCD}]
Let \(\beta\) be a positive constant and  \(B=\{X\in S_{++}\big|\beta I \preceq X \preceq \gamma I\}\), and assume that for every iteration \(k\), \(A^{(k)}\in B\). Let \(\{A^{(k_j)}\}\) be any infinite and converging sub-series of \(\{A^{(k)}\}\) and let \(\bar{A}\) denote its limit. Then \(\bar{A}=\argmin_{A\succ0}F(A)\).
\end{lemma}

\begin{proof}
Since the sub-series \(\{A^{(k_j)}\}\) convergences to \(\bar{A}\), then \(\{F(A^{(k_j)})\}\) convergences to \(F(\bar{A})\). According to \cref{lemma4.1}, the full series \(\{F(A^{(k)})\}\) is monotone and hence convergent because \(\{F(A^{(k_j)})\}\) is convergent. Therefore, \(\{F(A^{(k_j+1)})-F(A^{(k_j)})\}\rightarrow0\), which implies following \cref{eq4.1} that \(||A^{(k_j)} - A^{(k_j+1)}||^2_F \rightarrow 0\) and \(\lim_{j\rightarrow \infty}A^{(k_j+1)}=\bar{A}\). According to pISTA, \(A^{(k+1)} = A^{(k)}+\mathcal{M}_A^{(k)}\odot D_t^{*(k)}\) and because \(\mathcal{M}_A^{(k)}\neq 0\), \(||A^{(k_j)} - A^{(k_j+1)}||^2_F \rightarrow 0\) implies \(||D_t^{*(k)}||^2_F\rightarrow 0\). We know that \(D_t^{*(k)}\) satisfies \cref{eq:d_grad2}, by taking the limit as \(j\rightarrow\infty\) and using \cref{lemma:fixed} we get that \(\bar{A}=\argmin_{A\succ0}F(A)\).
\end{proof}

\begin{lemma} \label{lemma:BCD4.3} [corresponds to lemma 4.3 on \cite{cov:MultiBCD}]
Let \(\beta\) be a positive constant and \(B=\{X\in S_{++}\big|\beta I \preceq X \preceq \gamma I\}\). Assume that for every iteration \(k\), \(A^{(k)}\in B\). Then the series \(\{A^{(k)}\}\) has a limit and it is given by \(\argmin_{A\succ0}F(A)\).
\end{lemma}
\begin{proof}
The domain \(B\) is compact, the rest of the proof follows the proof of lemma 4.3 in \cite{cov:MultiBCD} exactly.
\end{proof}
We have shown that pISTA converges under the condition of bounded eigenvalues. It remains to show that pISTA satisfies it.
\begin{lemma} [corresponds to lemma 2 on \cite{cov:QUIC}]
Let \(U=\{A | F(A)\leq F(A_0)\text{ and }A\in S_{++}\}\). Then,
 \(U\subseteq B=\{X\in S_{++}\big|\beta I \preceq X \preceq \gamma I\}\) where \(\beta>0\).
\end{lemma}
\begin{proof}
See lemma 2 in \cite{cov:QUIC}.
\end{proof}
\begin{corollary} \label{coro:compact}
The series \(\{A^{(k)}\}\) created by the pISTA algorithm satisfies \cref{t_cond_min}, hence, \(A\) is contained in \(U\) and in \(B\).
\end{corollary}

\section{Numerical Results}
\label{sec:numer}
We evaluate the performance of our pISTA algorithm using two setups---one using the GPU implementation and one using the CPU one.

The first evaluation is done by comparing the performance of pISTA to the performance of GISTA \cite{cov:GISTA}, Orthant-Based Newton (OBN) \cite{cov:newton_like} and ALM \cite{cov:ALM} (without skipping step\footnote{As implemented by the authors: \url{https://www.math.ucdavis.edu/~sqma/ALM-SICS.html}}), all of which were implemented efficiently and easily on GPU based on matrix operations. We note that Newton-Lasso \cite{cov:newton_like}, VSM \cite{cov:VSM} and PSM \cite{cov:PSM} can be implemented in the same manner on the GPU as well, but all of them have worse performance than OBN \cite{cov:newton_like} and ALM \cite{cov:ALM} according to the authors' measurements. Other algorithms, such as SQUIC \cite{bollhofer2019large}, BIG\&QUIC \cite{cov:BIG_QUIC} and BCD-IC \cite{cov:BCD}, can be implemented on GPU as well. However, their implementations, using GPU public math library, might perform worse than the CPU counterparts due to the special routines needed in these algorithms (e.g., sparse Hessian-vector multiplication). Optimal implementations are feasible but require a deep understanding of the GPU parallelism model and hardware and, thus, are highly non-trivial and complex. The algorithms are all implemented as recommended by their authors in Python, using 32-bit floating-point precision. Each algorithm uses the GPU as efficiently as possible with the CuPy library. The implementation does not utilize multi-threading explicitly. Instead, we rely on the mathematical operations of the CuPy, NumPy, or SciPy libraries and their internal parallelism.

The second evaluation is done by comparing the pISTA CPU-only implementation to the publicly available packages of SQUIC \cite{bollhofer2019large} and BIG\&QUIC \cite{cov:BIG_QUIC}.

In the first evaluation we initialize all methods using GISTA \cite{cov:GISTA} initialization where:
\begin{equation}
    A^{(0)}_{i,j} = \bigg\{\begin{array}{lr}
         (S_{i,i}+\alpha)^{-1}&  i=j\\
         0 & i\neq j
    \end{array}
\,.\end{equation}
On the second evaluation, we use GISTA \cite{cov:GISTA} initialization for pISTA only and use the default initialization implemented in the public packages themselves.

For ALM only, we consider any number which is less than or equal to \(10\cdot \epsilon_{machine}\) as zero and set those numbers to zero at the end of each iteration. For OBN we use ten inner iterations, and for ALM we use the hyperparameters used by \cite{cov:ALM} except that we set \(\mu_0=\frac{1}{\alpha}\) if \(\alpha < 0.5\). For our pISTA algorithm, we limit our linesearch over \(t\) by stopping if \(t\) is less than \(10^{-4}\). In that case, we use a step of \(t\leq\left(\frac{0.9}{cond(A^{(k)})}\right)^2\) which keeps the matrix positive definite. As proved in \cref{sec:analysis_min}, we can use a \(t<(\frac{\beta}{\gamma})^2\) where \(\beta\) and \(\gamma\) are lower and upper bounds of the eigenvalues, respectively. However, we cannot compute those bounds, so we assume that \(\left(\frac{0.9}{cond(A^{(k)})}\right)^2\) is sufficient and it should satisfy the linesearch criteria. For SQUIC, we follow the authors \cite{bollhofer2019large} and set the inverse tolerance \(\tau_{inv}\) to the same as the stopping criterion tolerance.

As a stopping criterion for all the methods, we follow \cite{cov:MultiBCD,cov:BIG_QUIC,cov:BCD} and use \(\min_z ||\partial F(A^{(k)})||_1 < \epsilon ||A^{(k)}||_1\), where \(\min_z ||\partial F(A^{(k)})||\) is the minimum sub-gradient norm, where \(\epsilon\) is set to \(10^{-2}\) and \(10^{-3}\) for the first and second evaluation, respectively.

The first evaluation's experiments were run on a machine with \textit{Intel(R) Xeon(R) Gold 6230 2.10GHz} processor, \textit{GeForce RTX 2080 Ti} GPU and \textit{Ubuntu 18.04.5} operating system, and were limited to \textit{4} cores and \textit{20GB} RAM. Also, we used \textit{Python3.9.2}, \textit{CUDA11.0}, \textit{Numpy1.20.1} and \textit{CuPy9.0}. The second evaluation's experiments were run on a machine with \textit{Intel(R) Xeon(R) Platinum 8362 2.80GHz} processor and \textit{Ubuntu 20.04.4} operating system, and were limited to \textit{128} cores and \textit{256GB} RAM. Also, We used \textit{Python3.8.10}, \textit{NumPy1.24.4} and \textit{SciPy1.10.1}.

Our full Python framework and code can be found at \newline 
\begin{center}
{\tt \url{https://github.com/GalSha/GLASSO_Framework}}.
\end{center}
\subsection{Evaluation of pISTA using the GPU Implementation}
\subsubsection{Synthetic Experiments} \label{sec:syn_gpu}
First, we evaluate the algorithms on synthetic data. We use three different types of matrices as our ground truth:
\begin{itemize}
    \item Chain graphs: as described in \cite{cov:QUIC}. The ground truth matrix \(\Sigma^{-1}\) is set to be \(\Sigma^{-1}_{i,i}=1\) and \(\Sigma^{-1}_{i,i+1}=\Sigma^{-1}_{i,i-1}=-0.5\).
    \item Graphs with random sparsity structures: as described in \cite{inexact}. We generate a sparse matrix \(U\) with non-zero elements set to be \(+1\) or \(-1\). We define the ground truth as \(\Sigma^{-1}=U^TU\) where all off diagonal elements are in \([-1,1]\). We control the number of non-zeros in \(U\) and tune it such that \(\Sigma^{-1}\) has approximately \(0.5\%\) non-zeros.
    \item Planar graphs: as described in \cite{cov:MultiBCD}. First, we create our graph \(G(V,E)\). Then, we generate \(n\) random points on the unit square to be our vertices \(V\). We use Delaunay triangulation to generate our edges \(E\). Given the graph connectivity, we define \(\Sigma^{-1}\) as its graph Laplacian, i.e., \(\Sigma^{-1}_{i,j}=-1\) if \((i,j)\in E\) and \(\Sigma^{-1}_{i,i}=d_i\) where \(d_i\) is the degree of vertex \(i\).
\end{itemize}
To ensure that the matrices are positive definite, we add a predefined diagonal term of \(\max\{-1.2\lambda_{min},10^{-1}\}\cdot I\).

We do two sets of experiments, with \(n=1,000\) and with \(n=10,000\). In each set, we draw \(3\%\cdot n\) samples and run the algorithms with two different values of \(\alpha\).
We repeat each experiment five times and show the average results in \cref{tab:results:1k} and \cref{tab:results:10k} for the first set and second set, respectively.
Also, we compare the results of a CPU-only implementation of our pISTA algorithm for both sets of experiments. The CPU counterpart implementation is implemented using the NumPy library in the same manner as the GPU implementation and relies on NumPy internal parallelism.  

In the tables, we show the average time and iterations it took for each algorithm to reach the stopping criterion. We also show the number of non-zeros in the output matrix and the minimum sub-gradient Frobenius norm.

\begin{table}
    \small\centering
    \resizebox{\textwidth}{!}{
 \begin{tabular}{ |c|c|c|c|c|c|c| } 
\hline
\multicolumn{2}{|c|}{Problem Parameters} & pISTA (GPU) & pISTA (CPU) & GISTA & OBN & ALM \\
\hline
\makecell{n\\ \(m/n\)\\ \(\alpha\)} & \makecell{\(\Sigma^{-1}\) type\\ \(|\Sigma^{-1}|_0\)} & \multicolumn{5}{|c|}{\makecell{time (iter)\\ \(\min_z||\partial F(\tilde{\Sigma}^{-1})||_F\)\\ \(||\tilde{\Sigma}^{-1}||_0\)}} \\
\hline\hline

\makecell{1,000\\ 3\%\\ 0.6} & \makecell{\(Chain\)\\ 2998} 
& \makecell{\textbf{0.04s (2.0)}\\ \textbf{0.0017}\\ 2959.2} 
& \makecell{0.35s (2.0)\\ 0.0017\\ 2959.2} 
& \makecell{0.05s (3.0)\\ 0.0275\\ 2970.4}
& \makecell{0.05s (2.0)\\ 0.0020\\ \textbf{2959.2}} 
& \makecell{0.55s (11.0)\\ 0.0312\\ 2966.0} \\
\hline
\makecell{1,000\\ 3\%\\ 0.4} & \makecell{\(Chain\)\\ 2998}  
& \makecell{\textbf{0.10s} (6.6)\\ 0.0142\\ 25307.2}
& \makecell{1.40s (6.6)\\ 0.0142\\ 25307.2} 
& \makecell{0.15s (10.6)\\ 0.0336\\ 25302.0}
& \makecell{0.11s (\textbf{4.0})\\ \textbf{0.0003}\\ 25246.0} 
& \makecell{2.47s (53.8)\\ 0.5966\\ 25246.8} \\
\hline
\makecell{1,000\\ 3\%\\ 0.6} & \makecell{\(Random\)\\ 5936} 
& \makecell{\textbf{0.04s (2.2)}\\ \textbf{0.0265}\\ 2184.0} 
& \makecell{0.32s (2.2)\\ 0.0265\\ 2184.0}
& \makecell{0.06s (3.6)\\ 0.0462\\ 2191.2}
& \makecell{0.06s (2.2)\\ 0.0824\\ 2185.2} 
& \makecell{0.61s (14.0)\\ 0.8561\\ 2209.0} \\
\hline
\makecell{1,000\\ 3\%\\ 0.4} & \makecell{\(Random\)\\ 5936}
& \makecell{\textbf{0.10s} (6.4)\\ 0.0666\\ 26335.2} 
& \makecell{1.05s (6.4)\\ 0.0666\\ 26335.2} 
& \makecell{0.21s (15.2)\\ 0.0286\\ 26400.4}
& \makecell{0.13s (\textbf{4.6})\\ \textbf{0.0005}\\ 26246.4} 
& \makecell{2.41s (53.4)\\ 2.5377\\ 26256.8} \\
\hline
\makecell{1,000\\ 3\%\\ 0.6} & \makecell{\(Planar\)\\ 6958}
& \makecell{\textbf{0.03s (2.0)}\\ \textbf{0.0008}\\ 2995.6}
& \makecell{0.31s (2.0)\\ 0.0008\\ 2995.6} 
& \makecell{0.04s (2.4)\\ 0.0269\\ 3011.6}
& \makecell{0.05s (2.0)\\ 0.0028\\ 2997.6} 
& \makecell{0.47s (10.2)\\ 0.3156\\ 3039.6} \\
\hline
\makecell{1,000\\ 3\%\\ 0.4} & \makecell{\(Planar\)\\ 6958}
& \makecell{0.24s (15.4)\\ 0.0244\\ 28495.2}
& \makecell{3.05s (15.6)\\ 0.0244\\ 28471.6} 
& \makecell{0.28s (20.2)\\ 0.0241\\ 28623.2}
& \makecell{\textbf{0.15s (5.6)}\\ \textbf{0.0130}\\ 28290.4} 
& \makecell{3.13s (69.6)\\ 1.8171\\ 28312.8} \\
\hline
\end{tabular}
}
   \caption{Results for \(1,000\times1,000\) precision matrix\\\small{Bold text marks the best results}}
    \label{tab:results:1k}
\end{table}

\begin{table}[h!]
    \small\centering
    \resizebox{\textwidth}{!}{
 \begin{tabular}{ |c|c|c|c|c|c|c| } 
\hline
\multicolumn{2}{|c|}{Problem Parameters} & pISTA (GPU) & pISTA (CPU) & GISTA & OBN & ALM \\
\hline
\makecell{n\\ \(m/n\)\\ \(\alpha\)} & \makecell{\(\Sigma^{-1}\) type\\ \(|\Sigma^{-1}|_0\)} & \multicolumn{5}{|c|}{\makecell{time (iter)\\ \(\min_z||\partial F(\tilde{\Sigma}^{-1})||_F\)\\ \(||\tilde{\Sigma}^{-1}||_0\)}} \\
\hline\hline

\makecell{10,000\\ 3\%\\ 0.4} & \makecell{\(Chain\)\\ 29,998}
& \makecell{\textbf{5.69s (3.0)}\\ 0.0101\\ 34685.2} 
& \makecell{249.18s (3.0)\\ 0.0101\\ 34685.2} 
& \makecell{10.80s (6.0)\\ 0.1478\\ 35104.0}
& \makecell{13.85s (3.0)\\ \textbf{0.0060}\\ 34839.2} 
& \makecell{103.44s (28.6)\\ 28.5176\\ 34879.6} \\
\hline
\makecell{10,000\\ 3\%\\ 0.2} & \makecell{\(Chain\)\\ 29,998}
& \makecell{\textbf{12.89s (7.0)}\\ 0.0951\\ 97765.6} 
& \makecell{578.37s (7.0)\\ 0.0951\\ 97765.6}
& \makecell{27.53s (16.6)\\ 0.8418\\ 98432.8}
& \makecell{20.02s (4.0)\\ \textbf{0.0129}\\ 97881.2} 
& \makecell{168.46s (48.0)\\ 24.0052\\ 98339.2} \\
\hline
\makecell{10,000\\ 3\%\\ 0.4} & \makecell{\(Random\)\\ 508,787}
& \makecell{\textbf{1.56s} (0.6)\\ \textbf{2.1558}\\ 10823.6} 
& \makecell{68.74s (0.6)\\ 2.1558\\ 10823.6}
& \makecell{1.85s (0.6)\\ 3.0901\\ 10823.6}
& \makecell{1.60s (0.6)\\ \textbf{2.1558}\\ 10823.6} 
& \makecell{767.79s (237.4)\\ 44.2140\\ 11072.0} \\
\hline
\makecell{10,000\\ 3\%\\ 0.2} & \makecell{\(Random\)\\ 508,787}
& \makecell{\textbf{5.67s (3.0)}\\ 4.3416\\ 140699.6} 
& \makecell{263.86s (3.0)\\ 4.3416\\ 140699.6} 
& \makecell{15.95s (9.0)\\ \textbf{0.8136}\\ 140869.6}
& \makecell{13.97s (3.0)\\ 1.0047\\ 140715.6} 
& \makecell{153.16s (44.2)\\ 21.1711\\ 141009.8} \\
\hline
\makecell{10,000\\ 3\%\\ 0.4} & \makecell{\(Planar\)\\ 69,949}
& \makecell{\textbf{3.92s (2.0)}\\ 1.2172\\ 63249.6} 
& \makecell{174.70s (2.0)\\ 1.2172\\ 63249.6} 
& \makecell{8.84s (5.0)\\ \textbf{0.0477}\\ 63446.8}
& \makecell{13.73s (3.0)\\ 0.1113\\ 63324.0} 
& \makecell{100.34s (28.2)\\ 6.6973\\ 63572.8} \\
\hline
\makecell{10,000\\ 3\%\\ 0.2} & \makecell{\(Planar\)\\ 69,949}
& \makecell{\textbf{20.91s (10.2)}\\ 1.4050\\ 194261.6} 
& \makecell{943.43s (10.2)\\ 1.4049\\ 194261.2} 
& \makecell{63.17s (35.8)\\ \textbf{0.5368}\\ 195714.0}
& \makecell{41.54s (7.6)\\ 1.3284\\ 192740.0} 
& \makecell{269.97s (79.4)\\ 12.4953\\ 193075.0}\\
\hline
\end{tabular}
}
   \caption{Results for \(10,000\times10,000\) precision matrix\\\small{Bold text marks the best results}}
    \label{tab:results:10k}
\end{table}

From \cref{tab:results:1k}, we see that pISTA outperforms every other algorithm with respect to running time on all the low-dimension matrices except one. We notice that pISTA requires fewer iterations than GISTA and more iterations than OBN. This is expected as GISTA is a first-order method, OBN is a second-order method, and the pISTA algorithm uses a relaxed second-order approximation. Thus, pISTA convergence rate should be between GISTA and OBN. For small values of \(\alpha\), OBN achieves a better minimal sub-gradient, however, the difference is negligible.
Also, the GPU implementation of pISTA is much faster than its CPU counterpart, proving the desirability of GPU-centric algorithms. 

In \cref{tab:results:10k}, we see similar results. pISTA outperforms every other algorithm with respect to running time on all matrices.  Also, as expected, pISTA requires fewer iterations than GISTA and more iterations than OBN. For high-dimension matrices, the speedup of the GPU over the CPU is more significant, making it extremely valuable in higher dimensions.

\begin{figure}[h]\begin{minipage}{\textwidth}
\renewcommand\footnoterule{}     
\centering
\begin{minipage}{.45\linewidth}
\centering
\subfloat[Chain precision matrix]{\label{fig:chn}\begin{tikzpicture}
\begin{semilogyaxis}[ymin=1e-4, height=.8\linewidth, grid=major,
xlabel={Iteration}, ylabel={\(F(\tilde{\Sigma}^{-1})-{min}F(A)\)}]
\addplot table [x=Iter,y=pISTA] {chain0.2.dat};
\addplot table [x=Iter,y=GISTA] {chain0.2.dat};
\addplot table [x=Iter,y=OBN] {chain0.2.dat};
\addplot table [x=Iter,y=ALM] {chain0.2.dat};
\end{semilogyaxis}
\end{tikzpicture}}
\end{minipage}
\begin{minipage}{.45\linewidth}
\centering
\subfloat[Random precision matrix]{\label{fig:rnd}\begin{tikzpicture}
\begin{semilogyaxis}[ymin=1e-4, height=.8\linewidth, grid=major,
xlabel={Iteration}, ylabel={\(F(\tilde{\Sigma}^{-1})-{min}F(A)\)}]
\addplot table [x=Iter,y=pISTA] {random0.2.dat};
\addplot table [x=Iter,y=GISTA] {random0.2.dat};
\addplot table [x=Iter,y=OBN] {random0.2.dat};
\addplot table [x=Iter,y=ALM] {random0.2.dat};
\end{semilogyaxis}
\end{tikzpicture}}
\end{minipage}\par\medskip
\centering
\begin{minipage}{.45\linewidth}
\centering
\subfloat[Planar precision matrix]{\label{fig:pln}\begin{tikzpicture}
\begin{semilogyaxis}[ymin=1e-4, height=.8\linewidth, grid=major,
xlabel={Iteration}, ylabel={\(F(\tilde{\Sigma}^{-1})-{min}F(A)\)},
legend entries={pISTA,GISTA,OBN,ALM}, legend style={legend columns=-1}, legend to name = plot_legend]
\addplot table [x=Iter,y=pISTA] {planar0.2.dat};
\addplot table [x=Iter,y=GISTA] {planar0.2.dat};
\addplot table [x=Iter,y=OBN] {planar0.2.dat};
\addplot table [x=Iter,y=ALM] {planar0.2.dat};
\end{semilogyaxis}
\end{tikzpicture}}
\end{minipage}\par\medskip
\ref{plot_legend}
\captionsetup{justification=centering,margin=2cm}
\caption{Semilog plot of \(F(\tilde{\Sigma}^{-1})-{min}F(A)\) at each iteration\protect\footnote{\centering The iterations are not equal in complexity or time between the algorithms} for \(10,000\times10,000\) matrix and \(\alpha=0.2\)}
\vspace{-4ex}
\label{fig:syn_plot}
\end{minipage}\end{figure}
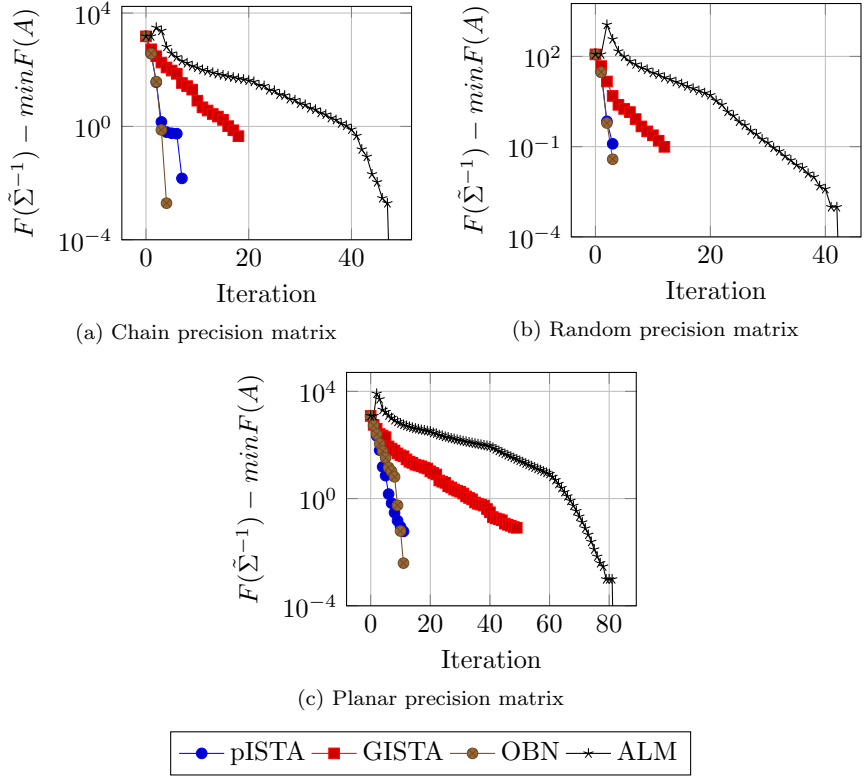

In \cref{fig:syn_plot}, we show a semi-log plot of \(F(\tilde{\Sigma}^{-1})-{min}F(A)\) as a function of the iteration, however, it is important to note that the iterations are not equal in complexity or required time. We only present the plots of the \(10,000\times10,000\) precision matrices for \(\alpha=0.2\) because of space consideration. We define \({min}F(A)\) as the minimum value achieved for \(F\) among all the algorithms and iterations in that experiment. The plots show what we expect, pISTA, which uses relaxed second-order approximation, achieves a convergence rate between linear (GISTA) and quadratic (OBN). Moreover, we see that the convergence rate is quadratic in the first few iterations.

\subsubsection{Real World Data Experiments}
For real-world data we use gene expression
data sets that are available at the Gene Expression Omnibus \url{http://www.ncbi.nlm.nih.gov/geo/}. We preprocess the data
to have zero mean and unit variance for each variable, i.e., \(diag(S) = I\). \cref{tab:results:gene} shows the results for data sets of various sizes, including the name codes of the data sets used.

\begin{table}
\small\centering
\resizebox{\textwidth}{!}{
 \begin{tabular}{ |c|c|c|c|c|c| } 
\hline
\multicolumn{2}{|c|}{Problem Parameters} & pISTA & GISTA & OBN & ALM \\
\hline
Data set & \makecell{n\\ \(m/n\)\\ \(\alpha\)} & \multicolumn{4}{|c|}{\makecell{time (iter)\\ \(\min_z||\partial F(\tilde{\Sigma}^{-1})||_F\)\\ \(||\tilde{\Sigma}^{-1}||_0\)}} \\
\hline\hline

GSE-3016 & \makecell{1322\\ 4.77\% (63)\\ 0.85} 
& \makecell{\textbf{0.09s (3)}\\ 0.1915\\ 3478} 
& \makecell{0.11s (3)\\ \textbf{0.0378}\\ 3560} 
& \makecell{0.16s (3)\\ 0.2185\\ 3490} 
& \makecell{1.01s (18)\\ 0.2921\\ 3596} \\
\hline

GSE-3016 & \makecell{1322\\ 4.77\% (63)\\ 0.75} 
& \makecell{\textbf{0.14s (5)}\\ 0.0092\\ 10796} 
& \makecell{0.35s (14)\\ 0.0049\\ 10890}
& \makecell{0.32s (6)\\ \textbf{0.0025}\\ 10806}  
& \makecell{2.06s (39)\\ 8.5051\\ 10809} \\
\hline

GSE-3016 & \makecell{1322\\ 4.77\% (63)\\ 0.65} 
& \makecell{1.02s (33)\\ 0.0134\\ 20682} 
& \makecell{1.55s (66)\\ 0.0158\\ 20660} 
& \makecell{\textbf{0.66s (11)}\\ 0.0016\\ 20386} 
& \makecell{3.17s (59)\\ \textbf{0.0011}\\ 20440} \\
\hline

GSE-26242 & \makecell{1536\\ 6.25\% (96)\\ 0.85} 
& \makecell{\textbf{0.08s (2)}\\ 0.0118\\ 4876} 
& \makecell{0.15s (4)\\ 0.0186\\ 4962} 
& \makecell{0.14s (2)\\ 0.0103\\ 4862} 
& \makecell{1.67s (26)\\ \textbf{0.0081}\\ 4912} \\
\hline

GSE-26242 & \makecell{1536\\ 6.25\% (96)\\ 0.75} 
& \makecell{\textbf{0.45s} (14)\\ 0.0003\\ 13896}  
& \makecell{0.70s (24)\\ 0.0056\\ 13900} 
& \makecell{0.47s \textbf{(7)}\\ \textbf{0.0001}\\ 13798}
& \makecell{2.39s (40)\\ 3.5156\\ 13810} \\
\hline

GSE-26242 & \makecell{1536\\ 6.25\% (96)\\ 0.65} 
& \makecell{1.81s (32)\\ 0.0096\\ 28494} 
& \makecell{2.43s (83)\\ 0.0148\\ 28942} 
& \makecell{\textbf{1.01s (13)}\\ \textbf{0.0087}\\ 28148} 
& \makecell{4.28s (66)\\ 6.7166\\ 28182} \\
\hline

GSE-7039 & \makecell{6264\\ 3.99\% (250)\\ 0.9}
& \makecell{\textbf{2.31s (4)}\\ 0.0397\\ 34286} 
& \makecell{8.52s (15)\\ 0.0798\\ 34364} 
& \makecell{5.92s (4)\\ \textbf{0.0014}\\ 34250}
& \makecell{74.19s (77)\\ 33.9381\\ 34266} \\
\hline

GSE-7039 & \makecell{6264\\ 3.99\% (250)\\  0.8} 
& \makecell{10.42s (14)\\ 0.0566\\ 49200} 
& \makecell{39.91s (69)\\ 0.1218\\ 50020} 
& \makecell{\textbf{9.39s (6)}\\ \textbf{0.0449}\\ 49054}
& \makecell{*s ($>$1000)\\ *\\ *} \\
\hline

GSE-7039 & \makecell{6264\\ 3.99\% (250)\\  0.7} 
& \makecell{41.87s (41)\\ 0.0206\\ 70700} 
& \makecell{101.95s (166)\\ 0.1302\\ 70820} 
& \makecell{\textbf{16.39s (10)}\\ \textbf{0.0012}\\ 70572} 
& \makecell{*s ($>$1000)\\ *\\ *} \\
\hline

GSE-52076 & \makecell{11064\\ 5.39\% (596)\\ 0.9} 
& \makecell{\textbf{9.78s (4)}\\ \textbf{0.0027}\\ 32840} 
& \makecell{30.65s (13)\\ 0.1908\\ 32850}
& \makecell{Out of GPU memory}  
& \makecell{103.10s (23)\\ 23.4233\\ 32990}\\
\hline

GSE-52076 & \makecell{11064\\ 5.39\% (596)\\ 0.8} 
& \makecell{\textbf{10.05s (4)}\\ \textbf{0.1162}\\ 42852} 
& \makecell{55.54s (23)\\ 0.2827\\ 43086}
& \makecell{Out of GPU memory}  
& \makecell{168.27s (39)\\ 71.0796\\ 43471}\\
\hline

GSE-52076 & \makecell{11064\\ 5.39\% (596)\\ 0.7} 
& \makecell{\textbf{17.34s (7)}\\ 0.3950\\ 49612}
& \makecell{167.93s (65)\\ \textbf{0.3866}\\ 49644}
& \makecell{Out of GPU memory}  
& \makecell{317.99s (75)\\ 58.4818\\ 49840}\\
\hline
\end{tabular}
}
   \caption{Results for real-world data set\\\small{Bold text marks the best results}}
    \label{tab:results:gene}
\end{table}

In \cref{tab:results:gene}, we see that there is no silver bullet. There are some problems where pISTA will outshine and there are some problems where OBN will outshine. However, OBN requires much more memory than pISTA, making it hard to use in high dimensions on GPUs.

\subsection{Evaluation of pISTA using the CPU Implementation}
As in \cref{sec:syn_gpu}, we evaluate the algorithms on synthetic data, and we use \textit{Planar graphs} and \textit{Graphs with random sparsity structures} as our ground truth. Similarly, to ensure that the matrices are positive definite, we add a predefined diagonal term of \(\max\{-1.2\lambda_{min},10^{-1}\}\cdot I\).

We do two sets of experiments, with \(n=4,000\) and with \(n=8,000\). In each set we draw \(3\%\cdot n\) samples and run the algorithms with two different values of \(\alpha\).
We repeat each experiment five times and show the average results in \cref{tab:results:48k}.
As said, we compare the results to a CPU-only implementation of our pISTA algorithm. However, we also compare the results to \(64\)-bit floating precision pISTA because the public packages are all implemented using \(64\)-bit floating precision.

\begin{table}
    \small\centering
    \resizebox{\textwidth}{!}{
 \begin{tabular}{ |c|c|c|c|c|c| } 
\hline
\multicolumn{2}{|c|}{Problem Parameters} & pISTA 32 bit (CPU) & pISTA 64 bit (CPU) & SQUIC (CPU) & BIG\&QUIC (CPU)\\
\hline
\makecell{n\\ \(m/n\)\\ \(\alpha\)} & \makecell{\(\Sigma^{-1}\) type\\ \(|\Sigma^{-1}|_0\)} & \multicolumn{4}{|c|}{\makecell{time (iter)\\ \(\min_z||\partial F(\tilde{\Sigma}^{-1})||_F\)\\ \(||\tilde{\Sigma}^{-1}||_0\)}} \\
\hline\hline

\makecell{4,000\\ 3\%\\ 0.5} & \makecell{\(Random\)\\ 83,701} 
& \makecell{4.35s (\textbf{2.0})\\ \textbf{0.0238}\\ 4890.4} 
& \makecell{6.04s (\textbf{2.0})\\\textbf{0.0238}\\ 4890.4} 
& \makecell{\textbf{0.49s} (4.4)\\ 0.0683\\ 4890.4} 
& \makecell{3.44s (3.6)\\ 0.0425\\ 4889.6} \\
\hline
\makecell{4,000\\ 3\%\\ 0.25} & \makecell{\(Random\)\\ 83,701}
& \makecell{\textbf{15.51s} (7.2)\\ 0.0282\\ 123622.0} 
& \makecell{23.41s (7.2)\\ 0.0282\\ 123622.0} 
& \makecell{29.18s (\textbf{5.8})\\ \textbf{0.0059}\\ 124274.0} 
& \makecell{115.48s (6.4)\\ 0.0064\\ 123592.0} \\
\hline
\makecell{4,000\\ 3\%\\ 0.5} & \makecell{\(Planar\)\\ 27,953}
& \makecell{4.34s (\textbf{2.0})\\ \textbf{0.0014}\\ 11822.0}
& \makecell{6.06s (\textbf{2.0})\\ \textbf{0.0014}\\ 11822.0}
& \makecell{\textbf{0.62}s (5.4)\\ 0.0107\\ 11831.6} 
& \makecell{9.18s (4.6)\\ 0.0139\\ 11822.8} \\
\hline
\makecell{4,000\\ 3\%\\ 0.25} & \makecell{\(Planar\)\\  27,953}
& \makecell{\textbf{49.38s} (19.8)\\ 0.0032\\ 131574.0}
& \makecell{79.49s (19.6)\\ 0.0031\\ 131592.0}
& \makecell{73.82s (6.8)\\ 0.0023\\ 131549.2}
& \makecell{173.59s (\textbf{6.2})\\ \textbf{0.0011}\\ 131200.8} \\
\hline
\makecell{8,000\\ 3\%\\ 0.4} & \makecell{\(Random\)\\ 327,450} 
& \makecell{20.99s (\textbf{2.0})\\ \textbf{0.0178}\\ 9808.4} 
& \makecell{29.04s (\textbf{2.0})\\ \textbf{0.0178}\\ 9808.4} 
& \makecell{\textbf{0.73s} (4.4)\\ 0.1022\\ 9810.0}
& \makecell{21.07s (3.8)\\ 0.1518\\ 9808.0} \\
\hline
\makecell{8,000\\ 3\%\\ 0.2} & \makecell{\(Random\)\\ 327,450}
& \makecell{\textbf{58.37s} (\textbf{5.4})\\ 0.0337\\ 216172.4} 
& \makecell{81.80s (\textbf{5.4})\\ 0.0319\\ 216165.6} 
& \makecell{105.65s (8.2)\\ \textbf{0.0200}\\ 216735.6} 
& \makecell{463.94s (6.2)\\ 0.0777\\ 216148.0} \\
\hline
\makecell{8,000\\ 3\%\\ 0.4} & \makecell{\(Planar\)\\ 55,945}
& \makecell{35.02s (\textbf{3.4})\\ 0.0387\\ 50227.6}
& \makecell{47.22s (\textbf{3.4})\\ 0.0387\\ 50227.6}
& \makecell{\textbf{1.52s} (6.0)\\ \textbf{0.0009}\\ 50292.8} 
& \makecell{71.43s (4.0)\\ 0.0020\\ 50247.2} \\
\hline
\makecell{8,000\\ 3\%\\ 0.2} & \makecell{\(Planar\)\\ 55,945}
& \makecell{327.07s (21.2)\\ 0.0249\\ 223343.2}
& \makecell{360.49s (18.6)\\ 0.0285\\ 223335.2}
& \makecell{\textbf{161.60s} (7.0)\\ 0.0049\\ 223800.0}
& \makecell{687.39s (\textbf{6.0})\\ \textbf{0.0032}\\ 223050.0} \\
\hline
\end{tabular}
}
   \caption{Results for public packages evaluation\\\small{Bold text marks the best results}}
    \label{tab:results:48k}
\end{table}

In \cref{tab:results:48k}, we see that the cases where \(\alpha\) is small are better solved (in terms of time) by our pISTA algorithm, and the cases where \(\alpha\) is large are better solved by SQUIC. These are cases where the obtained matrix is sparser, so exploiting the sparsity is more beneficial. Moreover, in all cases except one, pISTA is better than BIG\&QUIC. However, it is important to note that those results are of CPU-only implementation. It is safe to assume, as shown on \cref{sec:syn_gpu}, that GPU implementation of pISTA will show better results by a notable margin. In any case, the results shown on \cref{tab:results:48k} reinforce the assertion that there is no silver bullet.

\subsection{The Influence of the Sample Ratio and \(\alpha\) Parameter}
In all of the previous experiments, we chose \(\alpha\) parameter such that it will create different sparsity levels in the final result. Choosing the best \(\alpha\), in terms of ground truth estimation, is dependent on both the actual ground truth and the number of samples available.  

We evaluate the effect of different values of \(\alpha\) and different numbers of samples on synthetic data with \(n=3,000\). We use the \textit{Chain}, \textit{Planar}, and \textit{Random} graphs as ground truth. As done in our previous experiments, to ensure that the matrices are positive definite, we add a predefined diagonal term of \(\max\{-1.2\lambda_{min},10^{-1}\}\cdot I\). We ran the experiments with the same stopping criteria as previously defined with \(\epsilon=10^{-2}\).

We repeat each experiment five times and show the average Matthews correlation coefficient achieved in \cref{fig:syn_mat}. Matthews correlation coefficient (MCC) or Phi coefficient is used to measure the quality of the estimated sparse pattern, and it is defined as:
\begin{equation}
    \text{MCC} = \dfrac{TP \times TN - FP \times FN}{\sqrt{(TP + FP)(TP + FN)(TN + FP)(TN + FN)}}
,\end{equation}
where $TP$ is the number of true positives, $TN$ the number of true negatives, $FP$ the number of false positives and $FN$ the number of false negatives.

\begin{figure}
    \centering
    \begin{minipage}[t]{.33\textwidth}
        \centering
\subfloat[Chain precision matrix]{\label{mat:chn}
\includegraphics[width=0.9\textwidth]{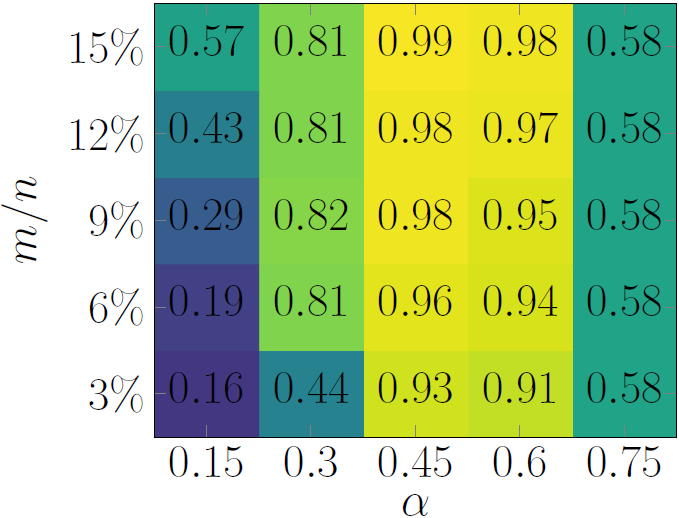}
}
    \end{minipage}\hfill
    \begin{minipage}[t]{.33\textwidth}
        \centering
\subfloat[Random precision matrix]{\label{mat:rnd}
\includegraphics[width=0.9\textwidth]{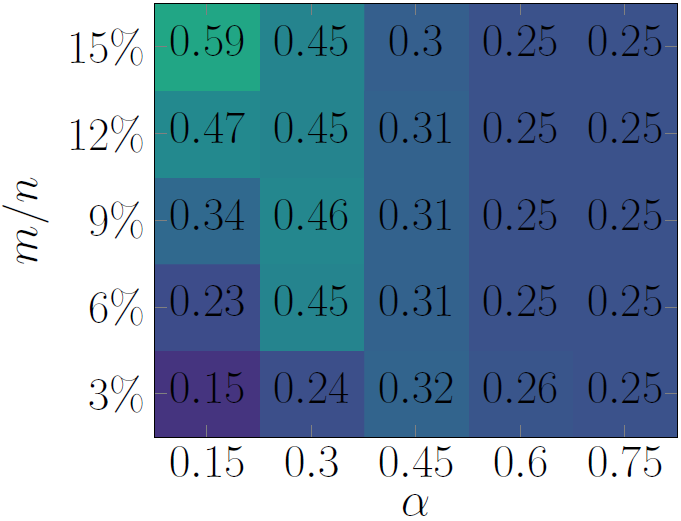}
}
    \end{minipage}\hfill
    \begin{minipage}[t]{.33\textwidth}
        \centering
\subfloat[Planar precision matrix]{\label{mat:pln}
\includegraphics[width=0.9\textwidth]{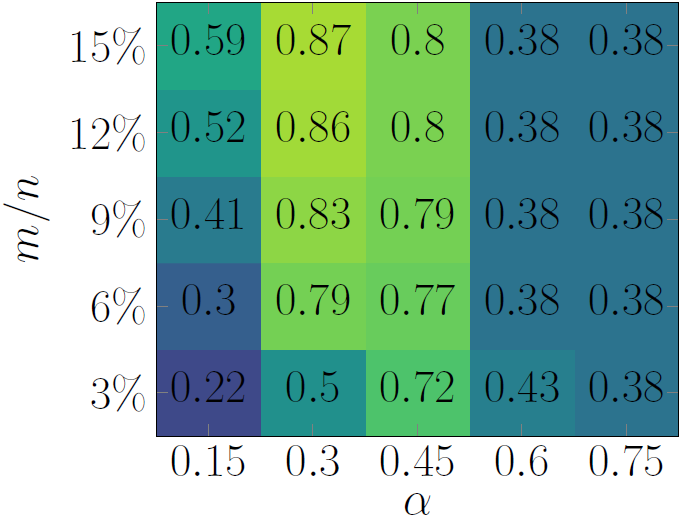}
}
    \end{minipage}\hfill
    \caption{\centering Matthews correlation coefficient as function of \(\alpha\) and number of samples for \(3000\times3000\) precision matrix}
    \label{fig:syn_mat}
\end{figure}

According to the results shown on \cref{fig:syn_mat}, the best value of \(\alpha\) is different on each type of ground truth. Also, as expected, as the number of samples increases, the quality of the sparsity pattern reconstruction is generally better. 

\section{Conclusions}
\label{sec:conclusions}
In this work we presented a method for solving sparse inverse covariance estimation problems. Our method creates simplified approximate second-order optimization where the different variables' dependencies can be guessed in a simple manner. Moreover, the method is designed to be implemented using matrix operations which can be done on a GPU, allowing us to solve the problem efficiently. We showed the desirability of GPU implementations and that our pISTA method has better results for various problem structures.

\newpage
\appendix
\appendixnotitle
\label{appendix:d_sol}
\subsection{General Solution for Sign Derivative Equation}

Define \(\mathcal{T}(x)\):
\begin{gather*} \label{T}
    x\neq0:\quad\quad \mathcal{T}(x)=sign(x)\\
    x=0:\quad\quad \mathcal{T}(x)\in[-1,1]
\,.\end{gather*}

\begin{lemma} \label{appendix:sol_1d_all}
The solution for:
\begin{equation}
    x+b+c\mathcal{T}(x+a)=0,\quad c>0
\end{equation}
is given by:
\begin{equation}
       x = \left\{\begin{array}{lr}
         -b+c & c < b-a\\
         -a & c > |b-a|\\ 
         -b-c & -c > b-a
    \end{array}\right. = -a + \text{SoftThreshold}(a-b, c)
\,.\end{equation}
\end{lemma}
\begin{proof}
First, consider the case where \(x>-a\Rightarrow\mathcal{T}(x+a)=1\):
\begin{gather*}
\left\{\begin{array}{lr}
        x+b+c = 0\\
        x > -a\\
        \end{array}\right.
\Rightarrow \left\{\begin{array}{lr}
        x = -b-c\\
        x > -a\\
        \end{array}\right.
\Rightarrow \left\{\begin{array}{lr}
        x = -b-c\\
        -b-c > -a\\
        \end{array}\right.
\Rightarrow \left\{\begin{array}{lr}
        x = -b-c\\
        -c > b-a\\
        \end{array}\right.
\,.\end{gather*}

In a similar way, for \(x<-a\Rightarrow\mathcal{T}(x+a)=-1\) and we get:
\begin{gather*}
\left\{\begin{array}{lr}
        x = -b+c\\
        c < b-a\\
        \end{array}\right.
\,.\end{gather*}
Lastly, consider the case that \(x=-a\), then since \(c > 0\) we get:
\begin{equation*}
    0\in -a+b+c\cdot t,\;\;t\in[-1,1]\\
\end{equation*}
which leads to 
\begin{gather*}
    a-b \in [-c,c] \Rightarrow c \leq |a-b| \Rightarrow c \leq |b-a|
\,.\end{gather*}

\end{proof}

\subsection{Approximate Solution for the pISTA Sub-gradient Equation}
Consider the following problem:
\begin{multline} 
    D:\quad\quad 0\in\Bigg\{
    \frac{1}{t}D +
    A\left(g \odot \mathcal{M}_A\right)A+\alpha A\left(\mathcal{G}\odot \mathcal{M}_A\right)A \\\Bigg| \begin{array}{lr}
         \mathcal{G}_{i,j}=[sign(A+\mathcal{M}_A \odot D)]_{i,j}&  [A+\mathcal{M}_A \odot D]_{i,j}\neq 0\\
         \mathcal{G}_{i,j}\in[-1,1]& [A+\mathcal{M}_A \odot D]_{i,j}=0
    \end{array} \Bigg\}
\,.\end{multline}
Let us split the equation into \(n^2\) scalar equations where the \((i,j)\)-th equation is composed of the elements with indices \((i,j)\). Denote the equation variable \(D_{i,j}=x\). The \((i,j)\) equation is:
\begin{equation} \label{appendix:g_1d_eq}
    \frac{1}{t}x+\left[A\left(g \odot \mathcal{M}_A\right)A\right]_{i,j}+\alpha\left[ A\left(\mathcal{G}\odot \mathcal{M}_A\right)A\right]_{i,j}=0
\,,\end{equation}
where \(\mathcal{G}_{i,j}=\mathcal{T}\left(A_{i,j}+[\mathcal{M}_A]_{i,j}\cdot x\right)\). As we consider equations where \([\mathcal{M}_A]_{i,j}\neq0\) (the rest are zeros), we can use  \(\mathcal{G}_{i,j}=\mathcal{T}\left(A_{i,j}+x\right)\). Define:
\begin{equation}\label{eq:G_comp}
    \mathcal{G}^{-i,-j}_{k,l}=\left\{\begin{array}{lr}
        0 & k=i, l=j\\
        0 & k=j, l=i\\
        \mathcal{G}_{k,l} & otherwise
    \end{array} \right.
\,.\end{equation}
We can write \cref{appendix:g_1d_eq} as:
\begin{equation} \label{appendix:g_1d_eq_2}
    \frac{1}{t}x+\left[A\left(g \odot \mathcal{M}_A\right)A\right]_{i,j}+\alpha\left[ A\left(\mathcal{G}^{-i,-j}\odot \mathcal{M}_A\right)A\right]_{i,j}+C_{i,j}\mathcal{T}\left(A_{i,j}+x\right)=0
\,,\end{equation}
where\begin{equation} \label{appendix:eq:c}
    C_{i,j}=\left\{\begin{array}{lc}
        \alpha\cdot\left(A_{i,i}\cdot A_{j,j}\right), & i= j\\
        \alpha\cdot\left(A_{i,i}\cdot A_{j,j}+A_{i,j}\cdot A_{j,i}\right), & i\neq j
    \end{array}\right.
\,\end{equation}
is the diagonal entry of the Kronecker matrix \(\alpha\cdot A\otimes A\) corresponding to the entry \((i,j)\). Note that \(C_{i,j} > 0\) since \(\alpha > 0\) and \(A^{(k)}\) is symmetric positive definite, and its diagonal is strictly positive. According to \cref{appendix:sol_1d_all}, the solution to \cref{appendix:g_1d_eq_2} is:
\begin{align*}
\label{eq:scalar_sol}
       x = -A_{i,j} +\text{SoftThreshold}\Big(A_{i,j}-&t\cdot\Big(\left[A\left(g \odot \mathcal{M}_A\right)A\right]_{i,j}\\
       +&\alpha\left[ A\left(\mathcal{G}^{-i,-j}\odot \mathcal{M}_A\right)A\right]_{i,j}\Big), t\cdot C_{i,j}\Big)
       \,.
\end{align*}
Notice that:
\begin{eqnarray}
    \alpha\left[ A\left(\mathcal{G}^{-i,-j}\odot \mathcal{M}_A\right)A\right]_{i,j} &=& \alpha\left[ A\left(\mathcal{G}\odot \mathcal{M}_A\right)A\right]_{i,j} - C_{i,j} \cdot \mathcal{G}_{i,j}\cdot[\mathcal{M}_A]_{i,j} \\ \nonumber&=&\left[ \alpha A\left(\mathcal{G}\odot \mathcal{M}_A\right)A - C \odot  \left(\mathcal{G}\odot \mathcal{M}_A\right)\right]_{i,j}
    \,,
\end{eqnarray}
where \(\odot\) is the Hadmard product. To write \eqref{eq:scalar_sol} more compactly, first define:
\begin{gather*}
    B = A\left(g \odot \mathcal{M}_A\right)A + \alpha A\left(\mathcal{G}\odot \mathcal{M}_A\right)A - C \odot \left(\mathcal{G}\odot \mathcal{M}_A\right)
    \,,
\end{gather*}
then we get
\begin{gather*}
       x = -A_{i,j} + \text{SoftThreshold}(A_{i,j}-t\cdot B_{i,j}, t\cdot C_{i,j})\\
       \Rightarrow D = -A + \text{SoftThreshold}(A-t\cdot B, t\cdot C)
       \,.
\end{gather*}

\bibliographystyle{siamplain}
\bibliography{references}

\begin{thebibliography}{10}

\bibitem{Sparse1}
{\sc O.~Banerjee, L.~El~Ghaoui, and A.~d'Aspremont}, {\em Model selection
  through sparse maximum likelihood estimation for multivariate gaussian or
  binary data}, J. Mach. Learn. Res.,  (2008).

\bibitem{Sparse2}
{\sc O.~Banerjee, L.~E. Ghaoui, A.~d'Aspremont, and G.~Natsoulis}, {\em Convex
  optimization techniques for fitting sparse gaussian graphical models}, in
  Proceedings of the 23rd International Conference on Machine Learning, 2006,
  p.~89–96.

\bibitem{FISTA}
{\sc A.~Beck and M.~Teboulle}, {\em A fast iterative shrinkage-thresholding
  algorithm for linear inverse problems}, SIAM Journal on Imaging Sciences,
  (2009), pp.~183--202.

\bibitem{bollhofer2019large}
{\sc M.~Bollhofer, A.~Eftekhari, S.~Scheidegger, and O.~Schenk}, {\em
  Large-scale sparse inverse covariance matrix estimation}, SIAM Journal on
  Scientific Computing,  (2019), pp.~A380--A401.

\bibitem{proof:derivation}
{\sc S.~Boyd and L.~Vandenberghe}, {\em Convex Optimization}, {Cambridge
  University Press}, 2004.

\bibitem{intro:other1}
{\sc X.~Chen, Y.~Liu, H.~Liu, and J.~Carbonell}, {\em Learning spatial-temporal
  varying graphs with applications to climate data analysis}, AAAI Conference
  on Artificial Intelligence,  (2010).

\bibitem{Sparse3}
{\sc A.~d'Aspremont, O.~Banerjee, and L.~Ghaoui}, {\em First-order methods for
  sparse covariance selection}, SIAM Journal on Matrix Analysis and
  Applications,  (2006).

\bibitem{prob:cov_select}
{\sc A.~P. Dempster}, {\em Covariance selection}, Biometrics,  (1972),
  pp.~157--175.

\bibitem{intro:bio}
{\sc A.~Dobra, C.~Hans, M.~Jones, J.~Nevins, G.~Yao, and M.~West}, {\em Sparse
  graphical models for exploring gene expression data}, Journal of Multivariate
  Analysis,  (2004), pp.~196--212.

\bibitem{cov:PSM}
{\sc J.~Duchi, S.~Gould, and D.~Koller}, {\em Projected subgradient methods for
  learning sparse gaussians}, in Uncertainty in Artificial Intelligence, 2008,
  p.~153–160.

\bibitem{intro:stocks}
{\sc J.~Fan, J.~Zhang, and K.~Yu}, {\em Vast portfolio selection with
  gross-exposure constraints}, Journal of the American Statistical Association,
   (2012), pp.~592--606.

\bibitem{finder2020effective}
{\sc S.~E. Finder, E.~Treister, and O.~Freifeld}, {\em Effective learning of a
  {GMRF} mixture model}, IEEE Access,  (2022).

\bibitem{cov:GLASSO}
{\sc J.~Friedman, T.~Hastie, and R.~Tibshirani}, {\em {Sparse inverse
  covariance estimation with the graphical lasso}}, Biostatistics,  (2007),
  pp.~432--441.

\bibitem{intro:social}
{\sc A.~Goldenberg and A.~W. Moore}, {\em Bayes net graphs to understand
  co-authorship networks?}, in Proceedings of the 3rd International Workshop on
  Link Discovery, 2005, p.~1–8.

\bibitem{cov:GISTA}
{\sc D.~Guillot, B.~Rajaratnam, B.~Rolfs, A.~Maleki, and I.~Wong}, {\em
  Iterative thresholding algorithm for sparse inverse covariance estimation},
  Advances in Neural Information Processing Systems,  (2012).

\bibitem{intro:sparse_model}
{\sc T.~Hastie, R.~Tibshirani, and M.~Wainwright}, {\em Statistical Learning
  with Sparsity: The Lasso and Generalizations}, Chapman \& Hall/CRC, 2015.

\bibitem{cov:QUIC}
{\sc C.-J. Hsieh, M.~A. Sustik, I.~S. Dhillon, and P.~Ravikumar}, {\em Quic:
  Quadratic approximation for sparse inverse covariance estimation}, Journal of
  Machine Learning Research,  (2014).

\bibitem{cov:BIG_QUIC}
{\sc C.-J. Hsieh, M.~A. Sustik, I.~S. Dhillon, P.~K. Ravikumar, and
  R.~Poldrack}, {\em Big \& quic: Sparse inverse covariance estimation for a
  million variables}, in Neural Information Processing Systems, 2013,
  pp.~3165--3173.

\bibitem{intro:other2}
{\sc T.~Idé, A.~C. Lozano, N.~Abe, and Y.~Liu}, {\em Proximity-Based Anomaly
  Detection using Sparse Structure Learning}, 2009, pp.~97--108.

\bibitem{intro:graph_model}
{\sc S.~Lauritzen}, {\em Graphical models}, Oxford Statistical Science Series,
  Clarendon Press, 1996.

\bibitem{inexact}
{\sc L.~Li and K.-C. Toh}, {\em An inexact interior point method for
  l1-regularized sparse covariance selection}, Mathematical Programming
  Computation,  (2010), p.~291–315.

\bibitem{cov:VSM}
{\sc Z.~Lu}, {\em Smooth optimization approach for sparse covariance
  selection}, SIAM Journal on Optimization,  (2009).

\bibitem{cov:newton_like}
{\sc P.~A. Olsen, F.~Oztoprak, J.~Nocedal, and S.~J. Rennie}, {\em Newton-like
  methods for sparse inverse covariance estimation}, in Neural Information
  Processing Systems, 2012, p.~755–763.

\bibitem{rue2005gaussian}
{\sc H.~Rue and L.~Held}, {\em Gaussian Markov random fields: theory and
  applications}, CRC press, 2005.

\bibitem{cov:ALM}
{\sc K.~Scheinberg, S.~Ma, and D.~Goldfarb}, {\em Sparse inverse covariance
  selection via alternating linearization methods}, in Neural Information
  Processing Systems, 2010, p.~2101–2109.

\bibitem{LASSO}
{\sc R.~Tibshirani}, {\em Regression shrinkage and selection via the lasso},
  Journal of the Royal Statistical Society. Series B (Methodological),  (1996),
  pp.~267--288.

\bibitem{cov:BCD}
{\sc E.~Treister and J.~S. Turek}, {\em A block-coordinate descent approach for
  large-scale sparse inverse covariance estimation}, in Neural Information
  Processing Systems (NIPS), 2014.

\bibitem{cov:MultiBCD}
{\sc E.~Treister, J.~S. Turek, and I.~Yavneh}, {\em A multilevel framework for
  sparse optimization with application to inverse covariance estimation and
  logistic regression}, SIAM Journal on Scientific Computing, 38 (2016),
  pp.~S566--S592.

\bibitem{treister2012multilevel}
{\sc E.~Treister and I.~Yavneh}, {\em A multilevel iterated-shrinkage approach
  to $ l\_ $\{$1$\}$ $ penalized least-squares minimization}, IEEE transactions
  on signal processing, 60 (2012), pp.~6319--6329.

\bibitem{proof:sparse_recon}
{\sc S.~J. Wright, R.~D. Nowak, and M.~A.~T. Figueiredo}, {\em Sparse
  reconstruction by separable approximation}, IEEE Transactions on Signal
  Processing,  (2009), pp.~2479--2493.

\bibitem{intro:stocks2}
{\sc X.~Xuan and K.~Murphy}, {\em Modeling changing dependency structure in
  multivariate time series}, in Proceedings of the 24th International
  Conference on Machine Learning, 2007, p.~1055–1062.

\bibitem{intro:sparse_assume}
{\sc M.~Yuan and Y.~Lin}, {\em Model selection and estimation in the gaussian
  graphical model}, Biometrika,  (2007), pp.~19--35.

\end{thebibliography}

\end{document}